\documentclass[reqno,12pt]{amsart}
\usepackage{amssymb}
\usepackage{amssymb, latexsym}
\usepackage{hyperref}
\usepackage{amsmath}
\usepackage{amscd}
\usepackage{enumerate}
\usepackage{amsfonts}
\usepackage{graphicx}
\usepackage[all]{xypic}
\usepackage{mathrsfs}

\usepackage{setspace}

\allowdisplaybreaks

\newcommand{\Z}{{\mathbb{Z}}}
\newcommand{\C}{{\mathbb{C}}}
\newcommand{\R}{{\mathbb{R}}}
\newcommand{\N}{{\mathbb{N}}}

\theoremstyle{plain}
\newtheorem{theorem}{Theorem}
\newtheorem{proposition}[theorem]{Proposition}
\newtheorem{lemma}[theorem]{Lemma}
\newtheorem{corollary}[theorem]{Corollary}

\newtheorem{conjecture}[theorem]{Conjecture}
\theoremstyle{definition}

\newtheorem{remark}[theorem]{Remark}

\numberwithin{equation}{section}
\numberwithin{theorem}{section}

\numberwithin{equation}{section}

\newcommand{\KKK}{\mathcal{K}}

\newcommand{\LLL}{\mathcal{L}}

\begin{document}

%  \doublespacing
  \onehalfspacing

\title[NLS with the combined terms]
{The dynamics of the 3D radial NLS with the combined terms}

\author[Miao]{Changxing Miao}
\address{\hskip-1.15em Changxing Miao:
\hfill\newline Institute of Applied Physics and Computational
Mathematics, \hfill\newline P. O. Box 8009,\ Beijing,\ China,\
100088,}
\email{miao\_changxing@iapcm.ac.cn}

\author[Xu]{Guixiang Xu}
\address{\hskip-1.15em Guixiang Xu \hfill\newline Institute of
Applied Physics and Computational Mathematics, \hfill\newline P. O.
Box 8009,\ Beijing,\ China,\ 100088, }
\email{xu\_guixiang@iapcm.ac.cn}

\author[Zhao]{Lifeng Zhao}
\address{\hskip-1.15em Lifeng Zhao \hfill\newline
University of Science and Technology of China, \hfill\newline
Hefei,\ China, } \email{zhaolifengustc@yahoo.cn}



\begin{abstract}In this paper, we show the scattering and blow-up result of the radial solution with the energy
below the threshold for the nonlinear Schr\"{o}dinger equation (NLS) with the combined terms
\begin{align*} iu_t +  \Delta u  =  -|u|^4u + |u|^2u \tag{CNLS}
\end{align*} in the energy space $H^1(\R^3)$. The threshold is given by the
ground state $W$ for the energy-critical NLS: $iu_t +  \Delta u  =
-|u|^4u$. This problem was proposed by Tao, Visan and Zhang in
\cite{TaoVZ:NLS:combined}. The main difficulty is the lack of the
scaling invariance. Illuminated by \cite{IbrMN:f:NLKG}, we need give the new radial profile decomposition with
the scaling parameter, then apply
it into the scattering theory. Our result shows that the defocusing,
$\dot H^1$-subcritical perturbation $|u|^2u$ does not affect the determination of the threshold
of the scattering solution of (CNLS) in the energy space.
\end{abstract}

\maketitle

%\tableofcontents

%%%%%%%%%%%%%%%%%%%%%%%%%%%%%%%%%%%%%%%%%%%%%%%%%%%%%%%%%%%%%%%%%%%%%%%%%%%%%%%%%%%%%%%%%%%
%
%
%                                   Section one
%
%
%%%%%%%%%%%%%%%%%%%%%%%%%%%%%%%%%%%%%%%%%%%%%%%%%%%%%%%%%%%%%%%%%%%%%%%%%%%%%%%%%%%%%%%%%%%

\section{Introduction}

We consider the dynamics of the radial  solutions for the
nonlinear Schr\"{o}dinger equation (NLS) with the combined
nonlinearities in $H^1(\R^3)$
\begin{equation} \label{NLS}
\left\{ \aligned
    iu_t +  \Delta u  = &\; f_1(u) + f_2(u),  , \quad (t,x)\in \R \times \R^3 , \\
     u(0)= & \; u_0(x)\in H^1(\R^3).
\endaligned
\right.
\end{equation}
where $u:\R \times \R^3 \mapsto \C$ and $f_1(u)=-|u|^4u$, $f_2(u)=
|u|^2u$. As we known, $f_1$ has the $\dot H^1$-critical growth, $f_2$ has the $\dot H^1$-subcritical growth.

The equation has the following mass and Hamiltonian quantities
\begin{align*}
M(u)(t)= & \frac12 \int_{\R^3} |u(t,x)|^2 \; dx; \quad  E(u)(t)=
\int_{\R^3} \frac12 |\nabla u(t,x)|^2 \; dx  + F_1(u(t)) + F_2(u(t))
\end{align*}
where $  F_1(u(t)) = \displaystyle  -\frac16 \int_{\R^3}
|u(t,x)|^6\; dx,\;\;  F_2(u(t))= \frac14  \int_{\R^3} |u(t,x)|^4 \;
dx.$ They are conserved for the sufficient smooth solutions of
\eqref{NLS}.

In \cite{TaoVZ:NLS:combined}, Tao, Visan and Zhang made the
comprehensive study of
\begin{align*}
iu_t +  \Delta u  =  |u|^4u + |u|^2u
\end{align*}
in the energy space. They made use of the interaction Morawetz
estimate established in \cite{CKSTT04} and the stability theory for the scattering solution. Their
result is based on the scattering result of the defocusing,
energy-critical NLS in the energy space, which is established by
Bourgain \cite{Bou:NLS:99, Bou:NLS:book} for the radial case, I-team
\cite{CKSTT07}, Ryckman-Visan \cite{RyV05} and Visan \cite{Vi05} for
the general data. Since the classical interaction Morawetz estimate in \cite{CKSTT04} fails
for \eqref{NLS}, Tao, et al., leave the scattering and blow-up
dichotomy of \eqref{NLS} below the threshold as an open problem in
\cite{TaoVZ:NLS:combined}. For other results, please refer to \cite{GinV:85:NLS, GinV:85:NLKG subcritical,
Nak:NLKG low dim:subcrit, Nak:01:NLKG subcritical, NakSch:cubic NLS:Rigidity,
Zha:global:NLKG, Zhang:NLS:06}.

For the focusing, energy-critical NLS
\begin{align}\label{NLS:focusing critical}
iu_t +  \Delta u = -|u|^4u.
\end{align}
Kenig and Merle first applied the concentration
compactness in \cite{BahG:NLW:proffile decomp, Ker:NLS:profile decomp, Ker:NLS:compactness}
into the scattering theory of the radial solution of
\eqref{NLS:focusing critical} in \cite{KenM:NLS:GWP} with the energy below that of the
ground state of
\begin{align}\label{ellip-critical}
-\Delta W = |W|^4W.
\end{align}
In this paper, we will also make use of the concentration compactness argument
and the stability theory to study the dichotomy of the radial solution of \eqref{NLS} with
the energy below the threshold, which will be shown to be the energy of the ground
state $W$ for \eqref{NLS:focusing critical}. For the applications of the concentration compactness in the
scattering theory and rigidity theory of the critical NLS, NLW, NLKG and Hartree
equations, please see \cite{Dod:NLS:higher dim, Dod:NLS:two dim,
Dod:NLS:one dim, Dod:NLS:foc,  DuyM:NLS:ThresholdSolution, DuyR:NLS:ThresholdSolution, IbrMN:f:NLKG, KenM:NLW:GWP,
KiTV:NLS:2d, KiV:en-NLS:high dim, KiVZ:NLS:high dim, LiZh:NLS, MiaoWX:Har:dynamic, MiXZ07b, MiXZ:NLS:radial NLS:mass
critical sca}.

We now show the differences between \eqref{NLS} and \eqref{NLS:focusing critical}. On one hand,
there is an explicit solution $W$ for \eqref{NLS:focusing critical}, which is the ground state of \eqref{ellip-critical}
and does not scatter. The threshold of the scattering solution of \eqref{NLS:focusing critical} is determined by the energy of
$W$. While for \eqref{NLS}, there is no such
explicit solution, whose energy is the threshold of the scattering solution of \eqref{NLS}. We need look for a mechanism to determine
the threshold of the scattering solution of \eqref{NLS}. It turns out that the constrained
minimization of the energy as \eqref{minimization} is appropriate\footnote{The similar constrained
minimization of the energy as \eqref{minimization} is not appropriate for the focusing perturbation: $iu_t +  \Delta u  =  -|u|^4u - |u|^2u$,
since the threshold $m$ in this way equals to $0$ and it is not the desired result.}.  On the other hand, for \eqref{NLS:focusing critical}, it is $\dot H^1$-scaling invariant,
which gives us many conveniences, especially in the nonlinear
profile decomposition about \eqref{NLS:focusing critical}. While for
\eqref{NLS}, it is the lack of scaling invariance. We need give the
new profile decomposition with the scaling parameter of
\eqref{NLS} in $H^1(R^3)$, take care of the role of the scaling
parameter in the linear and nonlinear profile decompositions, then apply them into
the scattering theory.

Now for $\varphi\in H^1$, we denote the scaling quantity
$\varphi^{\lambda}_{3, -2}$ by
\begin{align*}
\varphi^{\lambda}_{3, -2} (x)= e^{3\lambda}\varphi(e^{2\lambda}x).
\end{align*}
We denote the scaling derivative of $E$ by $K(\varphi)$
\begin{align}\label{scaling deriv:special}
K(\varphi)=  \LLL E(\varphi)
:=  \dfrac{d}{d \lambda } \Big|_{\lambda =0 } E(
\varphi^{\lambda}_{3, -2})  =  \int_{\R^3} \left( \frac{4}{2} |\nabla
\varphi|^2 - \frac{12}{6}|\varphi|^6 +\frac64 |\varphi|^4 \right)\; dx,
\end{align}
which is connected with the Virial identity, and then plays the
important role in the blow-up and scattering of the solution of
\eqref{NLS}.

Now the threshold $m$ is determined by the following constrained
minimization\footnote{In fact, the following minimization of the static energy
$$ \inf \{ M(\varphi)+ E(\varphi)\; |\;  \varphi \in H^1(\R^3),\; \varphi\not=0, \;
K(\varphi)=0 \}$$
 also equals to $m$.
} of the energy $E(\varphi)$
\begin{align}\label{minimization}
m = \inf \{ E(\varphi)\; |\;  \varphi \in H^1(\R^3),\; \varphi\not=0, \;
K(\varphi)=0 \}.
\end{align}
Since we consider the $\dot H^1$-critical growth with the $\dot H^1$-subcritical perturbation, we will use the modified energy later
\begin{align*}
E^c(u)= & \int_{\R^3} \left(\frac12 |\nabla u(t,x)|^2 -\frac16
|u(t,x)|^6 \right) \; dx.
\end{align*}

As the nonlinearity $|u|^2u$ is the defocusing, $\dot H^1$-subcritical perturbation, one think that the focusing, $\dot H^1$-critical term
plays the decisive role of the threshold of the scattering solution of \eqref{NLS} in the energy space.
The first result is to characterize the threshold energy $m$ as following

\begin{proposition}\label{threshold-energy} There is no minimizer for \eqref{minimization}. But for the threshold energy
$m$, we have
\begin{align*}
m = E^c(W),
\end{align*}
where $W\in \dot H^1(\R^3)$ is the ground state of the massless equation
\begin{align*}
-\Delta W = |W|^4W.
\end{align*}
\end{proposition}

As the dynamics of the solution of \eqref{NLS} with the energy less than the threshold $m$, the conjecture is
\begin{conjecture}\label{conjecture}
Let $u_0\in H^1(\R^3) $ with
\begin{equation}
E(u_0) <  m,
\end{equation}
and $u$ be the solution of \eqref{NLS} and $I$ be its maximal
interval of existence. Then
\begin{enumerate}
\item[\rm (a)] If $K(u_0)\geq 0$, then $I=\R$, and $u$ scatters in both time directions as $t\rightarrow \pm \infty$ in
$H^1$;

\item[\rm (b)] If $K(u_0)<0$, then $u$ blows up both
forward and backward at finite time in $H^1$.
\end{enumerate}
\end{conjecture}

In this paper, we verify the conjecture in the radial case.
\begin{theorem}\label{theorem}
Conjecture \ref{conjecture} holds whenever $u$ is spherically symmetric.
\end{theorem}

\begin{remark}
Our consideration of the radial case is based on the following facts:
\begin{enumerate}
\item It is an open problem that the scattering result of \eqref{NLS:focusing critical}
in dimension three, except for the radial case in \cite{KenM:NLS:GWP}.
Our result is based on the corresponding scattering result of \eqref{NLS:focusing critical}.
\item It seems to be hard to lower the regularity of the critical element to $L^{\infty}\dot H^s$ for some $s<0$ by the double Duhamel argument
in dimension three
to obtain the compactness of the critical element in $L^2$, which is used to control the spatial center function $x(t)$ of the critical element.
\end{enumerate}
%\item We will consider the effect of the focusing, $\dot H^1$-subcritical perturbation $-|u|^2u$ in future. As a focusing perturbation, we will
%see that this kind of perturbation will affect the determination of the threshold of the scattering solution in the energy space. It will lower the
%threshold.
%\end{enumerate}
\end{remark}

\begin{remark}
We can remove the radial assumption under the stronger constraint that
\begin{align*}
M(u_0)+E(u_0) < m,
\end{align*}
which can help us to obtain the compactness of the critical element in $L^2$ and control the spatial center function $x(t)$ of the critical
element. Of course, we need the precondition\footnote{By the relation between the sharp Sobolev constant and the ground state $W$, we know that the constrained condition
\begin{align*}
\int_{\R^3} \left( \big|\nabla u_0 \big|^2
-\big|u_0\big|^6 \right) \; dx \geq  0,\quad \int_{\R^3} \left( \frac12\big|\nabla u_0 \big|^2 - \frac16
  \big| u_0 \big|^6 \right)\; dx <  E^c(W)
\end{align*}
is equivalent to the constrained condition
\begin{align*}
 \big\|\nabla u_0 \big\|^2_{L^2}\leq \big\|\nabla W\big\|^2_{L^2},\quad \int_{\R^3} \left( \frac12\big|\nabla u_0 \big|^2 - \frac16
  \big| u_0 \big|^6 \right)\; dx < E^c(W).
\end{align*}
We use the former in this paper while the latter is given by Kenig-Merle in \cite{KenM:NLS:GWP}.} that the global wellposedness and scattering result of \eqref{NLS:focusing critical} holds
for $u_0\in \dot H^1(\R^3)$  with
\begin{align*}
 \int_{\R^3} \left( \big|\nabla u_0 \big|^2
-\big|u_0\big|^6 \right) \; dx \geq  0, \quad
 \int_{\R^3} \left( \frac12\big|\nabla u_0 \big|^2 - \frac16
  \big| u_0 \big|^6 \right)\; dx < & m.
\end{align*}
\end{remark}

\begin{remark}\label{rem:noempty}
From the assumption in Theorem \ref{theorem}, we know that the solution starts from the following subsets of the energy space,
\begin{align*}
\KKK^{+}=& \Big\{ \varphi \in H^1(\R^3)\; \;\Big|\;\; \varphi\; \text{is radial},\; E(\varphi)<m,\; K(\varphi)\geq 0  \Big\},\\
\KKK^{-}=& \Big\{ \varphi \in H^1(\R^3)\; \;\Big|\;\; \varphi\; \text{is radial},\; E(\varphi)<m,\; K(\varphi)< 0  \Big\}.
\end{align*}
By the scaling argument, we know that $\KKK^{\pm} \not= \emptyset$ (we can also know that $\KKK^+ \not = \emptyset$ by the small data theory). In fact, let $\chi(x)$ be a radial smooth cut-off function satisfying
$0\leq \chi \leq 1$, $\chi (x)=1$ for $|x|\leq 1$ and $\chi (x)=0$ for $|x|\geq 2$. If we take $\chi_R(x)=\chi(x/R)$ and
\begin{align*}
\varphi(x)=\theta \lambda^{-1/2}\chi_R(x/\lambda)W(x/\lambda),
\end{align*}
where $\theta, \lambda, R$ is determined later and the cutoff function $\chi_R$ is not needed for dimension $d\geq 5$ since $W\in H^1$.
Then we have
\begin{align*}
\big\|\nabla\varphi\big\|^2_{L^2}= & \theta^2 \left( \big\|\nabla W\big\|^2_{L^2} + \int \left((\chi_R^2-1)\big|\nabla W\big|^2 + |\nabla \chi_R|^2 |W|^2 + 2\chi_R \nabla\chi_R \cdot W \nabla W \right)\; dx\right),\\
\big\| \varphi\big\|^6_{L^6}=& \theta^6 \left( \big\|W\big\|^6_{L^6} + \int (\chi_R^6 -1) |W|^6 \; dx\right),\quad
\big\| \varphi\big\|^4_{L^4}=  \lambda \cdot \theta^4 \big\|\chi_R W\big\|^4_{L^4},\\
\big\| \varphi\big\|^2_{L^2}=& \lambda^2 \cdot \theta^2 \big\|\chi_RW\big\|^2_{L^2}.
\end{align*}
Therefore, taking $R$ sufficiently large, $\theta=1+\epsilon$ and $\lambda=\epsilon^3$  , we have
\begin{align*}
E(\varphi)= & \frac{\theta^2}{2}\big\|\nabla W\big\|^2_{L^2} -\frac{\theta^6}{6}\big\|W\big\|^6_{L^6}\\
& + \frac{\theta^2}{2} \int \left((\chi_R^2-1)\big|\nabla W\big|^2 + |\nabla \chi_R|^2 |W|^2 + 2\chi_R \nabla\chi_R \cdot W \nabla W \right)\; dx \\
& - \frac{\theta^6}{6} \int (\chi_R^6 -1) |W|^6 \; dx  + \lambda\cdot \frac{\theta^4}{4} \big\|\chi_R W\big\|^4_{L^4}\\
=& m - 6 \epsilon^2 m + o(\epsilon^2), \\
K(\varphi)= & 2 \theta^2 \big\|\nabla W\big\|^2_{L^2} - 2\theta^6 \big\|W\big\|^6_{L^6}\\
 & + 2 \theta^2  \int \left((\chi_R^2-1)\big|\nabla W\big|^2 + |\nabla \chi_R|^2 |W|^2 + 2\chi_R \nabla\chi_R \cdot W \nabla W \right)\; dx \\
 &  - 2 \theta^6 \int (\chi_R^6 -1) |W|^6 \; dx+ \lambda\cdot \frac{3\theta^4}{2} \big\|\chi_R W\big\|^4_{L^4}\\
 = & -24 \epsilon m + o(\epsilon^2).
\end{align*}
If taking $\epsilon < 0$ and $|\epsilon|$ sufficient small, then we have $\varphi \in \KKK^+$; If taking $\epsilon > 0$ and sufficient small, then we have
$\varphi \in \KKK^-$.
\end{remark}

\subsection*{Acknowledgements.}The authors are partly supported by the NSF
of China (No. 10801015, No. 10901148, No. 11171033). The authors
would like to thank Professor K.~Nakanishi for his valuable
communications. \qed

%%%%%%%%%%%%%%%%%%%%%%%%%%%%%%%%%%%%%%%%%%%%%%%%%%%%%%%%%%%%%%%%%%%%%%%%%%%%%%%%%%%%%%%%%%%
%
%
%                                   Section two
%
%
%%%%%%%%%%%%%%%%%%%%%%%%%%%%%%%%%%%%%%%%%%%%%%%%%%%%%%%%%%%%%%%%%%%%%%%%%%%%%%%%%%%%%%%%%%%

\section{Preliminaries}
In this section, we give some notation and some wellknown results.

\subsection{Littlewood-Paley decomposition and Besov space}

Let $\Lambda_0(x) \in \mathcal{S}(\R^3)$ such that its Fourier
transform $\widetilde{\Lambda}_0(\xi) =1$ for $|\xi|\leq 1$ and
$\widetilde{\Lambda}_0(\xi) = 0 $ for $|\xi|\geq 2$. Then we define
$\Lambda_k(x)$ for any $k\in \Z\backslash \{0\}$ and $\Lambda_{(0)}(x)$ by the
Fourier transforms:
\begin{align*}
\widetilde{\Lambda}_k(\xi)=\widetilde{\Lambda}_0(2^{-k}\xi) -
\widetilde{\Lambda}_0(2^{-k+1}\xi), \quad
\widetilde{\Lambda}_{(0)}(\xi)=\widetilde{\Lambda}_0(\xi)-\widetilde{\Lambda}_0(2\xi).
\end{align*}
Let $s\in \R$, $1\leq p, q \leq \infty$. The inhomogeneous Besov space $B^s_{p,q}$ is defined by
\begin{align*}
B^s_{p,q}=\left\{f \; \big|\; f\in \mathcal{S}', \Big\|2^{ks}\big\|\Lambda_k*f\big\|_{L^p_x}\Big\|_{l^q_{k\geq 0}}<\infty\right\},
\end{align*}
where $\mathcal{S}'$ denotes the space of tempered distributions. The homogeneous Besov space $\dot B^s_{p,q}$ can be
defined by
\begin{align*}
\dot B^s_{p,q}=\left\{f \; \Big|\; f\in \mathcal{S}', \left( \sum_{k\in\Z\backslash \{0\}}  2^{qks}\big\|\Lambda_k*f\big\|^q_{L^p_x}  + \big\|\Lambda_{(0)}*f\big\|_{L^p_x}\Big\|^q\right)^{1/q}<\infty\right\}.
\end{align*}

\subsection{Linear estimates} We say that a pair of exponents $(q,r)$ is Schr\"{o}idnger $\dot H^s$-admissible in dimension three if
\begin{align*}
\dfrac2q+\dfrac{3}{r}=\dfrac{3}{2}-s
\end{align*}
and $2\leq q, r \leq \infty$. If $I\times \R^3$ is a space-time
slab, we define the $\dot S^0(I\times \R^3)$ Strichartz norm by
\begin{align*}
\big\| u\big\|_{\dot S^0(I\times \R^3)}:=\sup\big\|u\big\|_{L^q_t
L^r_x(I\times \R^3)}
\end{align*}
where the sup is taken over all $L^2$-admissible pairs $(q,r)$. We define
the $\dot S^s(I\times \R^3)$ Strichartz norm to be
\begin{align*}
\big\| u\big\|_{\dot S^s(I\times \R^3)}:=\big\| D^s u\big\|_{\dot
S^0(I\times \R^3)}.
\end{align*}
We also use $\dot N^0(I\times
\R^3)$ to denote the dual space of $\dot S^0(I\times \R^3)$ and
\begin{align*}
\dot N^k(I\times \R^3):=\{u; D^k u \in \dot N^0(I\times \R^3)\}.
\end{align*}

By definition and Sobolev's inequality, we have
\begin{lemma} For any $\dot S^1$ function $u$ on $I\times \R^3$, we have
\begin{align*}
\big\|\nabla u\big\|_{L^{\infty}_tL^2_x}+\big\|u\big\|_{L^{10}_t
 \dot B^{1/3}_{90/19, 2}(I\times \R^3)}  +\big\|u\big\|_{L^{\infty}_tL^{6}_x}+ \big\|u\big\|_{L^{12}_tL^{9}_x}+
\big\|u\big\|_{L^{10}_{t,x}} \lesssim \big\|u \big\|_{\dot S^1}.
\end{align*}
For any $\dot S^{1/2}$ function $u$ on $I\times \R^3$, we have
\begin{align*}
\big\| u\big\|_{L^{\infty}_t\dot H^{1/2}_x}+\big\|u\big\|_{L^{6}_t
 \dot B^{1/2}_{18/7, 2}(I\times \R^3)}  + \big\|u\big\|_{L^{\infty}_tL^{3}_x}+ \big\|u\big\|_{L^{6}_tL^{9/2}_x}+
\big\|u\big\|_{L^{5}_{t,x}} \lesssim \big\|u \big\|_{\dot S^{1/2}}.
\end{align*}
\end{lemma}

Now we state the standard Strichartz estimate.

\begin{lemma}[\cite{Caz:NLS:book, KeT98, tao:book}]\label{Strichartz }
Let $I$ be a compact time interval, $k\in [0,1]$, and let $u: I\times
\R^3\rightarrow \C$ be an $\dot S^k$ solution to the forced
Schr\"{o}dinger equation
\begin{align*}
iu_t + \Delta u =F
\end{align*}
for a function $F$. Then we have
\begin{align*}
\big\|u\big\|_{\dot S^k(I\times \R^3)} \lesssim
\big\|u(t_0)\big\|_{\dot H^k(\R^d)} + \big\|F\big\|_{\dot
N^k(I\times \R^3)},
\end{align*}
for any time $t_0\in I$.
\end{lemma}

We shall also need the following exotic Strichartz estimate, which
is important in the application of the stability theory.

\begin{lemma}[\cite{Fos:NLS:exotic}]\label{strichartz:exotic} For any $F\in L^2_t \left(I;\dot
B^{1/3}_{18/11, 2} \right)$, we have
\begin{align*}
\left\|\int^t_0e^{i(t-s)\Delta} F(s)\; ds \right\|_{L^{10}_t \dot
B^{1/3}_{90/19, 2}} \lesssim \big\| F\big\|_{L^2_t \dot
B^{1/3}_{18/11, 2}}.
\end{align*}
\end{lemma}
%
%\subsection{Nonlinear estimates}
%

\subsection{Local wellposedness and Virial identity}

Let $$ST(I):= L^{10}_t \dot B^{1/3}_{90/19, 2} \cap L^{12}_tL^9_x \cap L^6_t\dot B^{1/2}_{18/7,2}\cap
L^{5}_{t,x}(I\times \R^3).$$
By the definition of admissible pair, we know that $ L^{10}_t \dot B^{1/3}_{90/19, 2} \cap L^{12}_tL^9_x$ is the $\dot H^1$-admissible space,
$L^6_t\dot B^{1/2}_{18/7,2}\cap
L^{5}_{t,x}$ is the $\dot H^{1/2}$-admissible space. Now we have

\begin{theorem}[\cite{TaoVZ:NLS:combined}]\label{lwp}
Let $u_0\in H^1$, then for every $\eta>0$, there exists $T=T(\eta)$
such that if
\begin{align*}
\big\|e^{it\Delta}u_0\big\|_{ST([-T, T])}\leq \eta,
\end{align*}
then \eqref{NLS} admits a unique strong $H^1_x$-solution $u$ defined
on $[-T, T]$. Let $(-T_{min}, T_{max})$ be the maximal time interval
on which $u$ is well-defined. Then, $u\in S^1(I\times \R^d)$ for
every compact time interval $I\subset (-T_{min}, T_{max})$ and the
following properties hold:
\begin{enumerate}
\item If $T_{max}<\infty$, then
\begin{align*}
\big\|u\big\|_{ ST((0, T_{max})\times \R^d)}=\infty.
\end{align*}
Similarly, if $T_{min}<\infty$, then
\begin{align*}
\big\|u\big\|_{ST((-T_{min}, 0)\times \R^d)}=\infty.
\end{align*}
\item The solution $u$ depends continuously on the initial data
$u_0$ in the following sense: The functions $T_{min}$ and $T_{max}$
are lower semicontinuous from $\dot H^1_x \cap \dot H^{1/2}_x$ to $(0, +\infty]$. Moreover,
if $u^{(m)}_0 \rightarrow u_0$ in $\dot H^1_x \cap \dot H^{1/2}_x$ and $u^{(m)}$ is the
maximal solution to \eqref{NLS} with initial data $u^{(m)}_0$, then
$u^{(m)}\rightarrow u$ in $ST(I\times \R^3)$ and every compact subinterval $I\subset (-T_{min}, T_{max})$.
\end{enumerate}
\end{theorem}

\begin{proof}
The proof is based on the Strichartz estimate and exotic Strichartz estimate and the following nonlinear estimates.
\begin{align*}
\big\||u|^4u\big\|_{L^2\dot B^{1/3}_{18/11,2}} \lesssim    \big\|u\big\|_{L^{10}_t\dot B^{1/3}_{90/19,2}}\big\|u\big\|^4_{L^{10}_{t,x}},\quad &
\big\||u|^2u\big\|_{L^2\dot B^{1/3}_{18/11,2}} \lesssim   \big\|u\big\|_{L^{10}_t\dot B^{1/3}_{90/19,2}}\big\|u\big\|^2_{L^{5}_{t,x}}, \\
\big\||u|^4u\big\|_{L^2\dot B^{1/2}_{6/5,2}} \lesssim   \big\|u\big\|_{L^{6}_t\dot B^{1/2}_{18/7,2}}\big\|u\big\|^4_{L^{12}_tL^9_{x}},\quad &
\big\||u|^2u\big\|_{L^2\dot B^{1/3}_{6/5,2}} \lesssim    \big\|u\big\|_{L^{6}_t\dot B^{1/2}_{18/7,2}}\big\|u\big\|^2_{L^{6}_tL^{9/2}_{x}}.
\end{align*}

\end{proof}

\begin{lemma}\label{L:virial}
Let $\phi\in C^{\infty}_0(\R^3)$, radially symmetric and $u$ be the radial solution of \eqref{NLS}. Then we have
\begin{align*}
\partial_t \int_{\R^3} \phi (x) \big| u(t,x)\big|^2\; dx
= &   -2 \Im \int_{\R^3}\nabla \phi \cdot \nabla \bar{u} \; u \; dx\\
\partial^2_t \int_{\R^3} \phi (x) \big| u(t,x)\big|^2\; dx
= &    4 \int_{\R^3} \phi''(r) \big|\nabla u \big|^2\;  dx
- \int_{\R^3} \Delta^2 \phi \big| u(t,x)\big|^2\; dx\\
& - \frac{4}{3}\int_{\R^3} \Delta \phi \big| u(t,x)\big|^6\; dx+\int_{\R^3} \Delta \phi \big| u(t,x)\big|^4\; dx,
\end{align*}
where $r=|x|$.
\end{lemma}
\begin{proof} By the simple computation, we have
\begin{align*}
\partial^2_t \int_{\R^3} \phi (x) \big| u(t,x)\big|^2\; dx = & 4 \int_{\R^3} \phi_{jk} \cdot \Re(u_k\overline{u}_j)\; dx
- \int_{\R^3} \Delta^2 \phi \cdot  \big| u(t,x)\big|^2\; dx\\
& - \frac{4}{3}\int_{\R^3} \Delta \phi \cdot  \big| u(t,x)\big|^6\; dx+\int_{\R^3} \Delta \phi \cdot \big| u(t,x)\big|^4\; dx.
\end{align*}
Then the result comes from the following fact
\begin{align*}
\partial^2_{jk}\phi(x)=\phi''(r)\frac{x_jx_k}{r^2}+\frac{\phi'(r)}{r}\left(\delta_{jk}-\frac{x_jx_k}{r^2}\right)
\end{align*}
holds for any radial symmetric function $\phi(x)$.
\end{proof}

\subsection{Variational characterization} In this subsection, we give the threshold energy $m$
(Proposition \ref{threshold-energy}) by the variational method,
and various estimates for the solutions of \eqref{NLS} with the energy below
the threshold. There is no the radial assumption on the solution.

We first give some notation before we show the behavior of $K$ near the origin. Let us denote the quadratic and nonlinear parts of $K$ by $K^Q$ and
$K^N$, that is,
\begin{align*}
K(\varphi)=K^Q(\varphi) + K^N(\varphi),
\end{align*}
where $K^Q(\varphi)= \displaystyle 2 \; \int_{\R^3} |\nabla \varphi|^2 \;
dx,$ and $K^N (\varphi)= \displaystyle \int_{\R^3} \left(-2 |  \varphi|^6 + \frac32  \varphi|^4 \right)\;
dx$.

\begin{lemma}\label{positive near origin:KQ }
For any $\varphi \in H^1(\R^3)$, we have
\begin{align}\label{asymptotic:KQ}
\lim_{\lambda \rightarrow -\infty}
K^Q(\varphi^{\lambda}_{3,-2}) =0.
\end{align}
\end{lemma}
\begin{proof}
It is obvious by the definition of $K^Q$.
\end{proof}

Now we show the positivity of $K$ near 0 in the energy
space.
\begin{lemma}\label{Postivity:K} For any bounded sequence
$\varphi_n\in H^1(\R^3) \backslash\{0\}$ with
\begin{align*}
\lim_{n\rightarrow +\infty}K^Q(\varphi_n)=0,
\end{align*}
then for large $n$, we have
\begin{align*}
K(\varphi_n)>0.
\end{align*}
\end{lemma}
\begin{proof} By the fact that $K^Q(\varphi_n) \rightarrow 0$, we know that
$\displaystyle
\lim_{n\rightarrow+\infty}\big\|\nabla\varphi_n\big\|^2_{L^2}=0.
$ Then by the Sobolev and Gagliardo-Nirenberg inequalities, we have for large
$n$
\begin{align*}
\big\|\varphi_n\big\|^6_{L^6_x} \lesssim   \big\|\nabla
\varphi_n\big\|^6_{L^2_x} = & o(\big\|\nabla\varphi_n\big\|^2_{L^2}),\\
\big\|\varphi_n\big\|^4_{L^4_x}\lesssim
\big\|\varphi_n\big\|_{L^2}\big\|\nabla \varphi_n\big\|^3_{L^2} & =
o(\big\|\nabla\varphi_n\big\|^2_{L^2}),
\end{align*}
where we use the boundedness of $\big\|\varphi_n\big\|_{L^2}$. Hence for large $n$, we have
 \begin{align*} K(\varphi_n)= &  \int_{\R^3}
\left( 2 |\nabla \varphi_n|^2 -2|\varphi_n|^6 +\frac32 |\varphi_n|^4
\right)\; dx \thickapprox   \int_{\R^3}  |\nabla \varphi_n|^2\; dx > 0.
\end{align*}
This concludes the proof.
\end{proof}

By the definition of $K$, we denote two real numbers by
%\begin{align*}\fbox{\rule[-3mm]{0cm}{10mm}$
%\bar{\mu} = \max\{4,0, 6\}=6, \quad \underline{\mu}=\min\{4,0,
%6\}=0.$}
%\end{align*}
\begin{align*}
\bar{\mu} = \max\{4,0, 6\}=6, \quad \underline{\mu}=\min\{4,0,
6\}=0.
\end{align*}

Next, we show the behavior of the scaling derivative functional $K$.

\begin{lemma}\label{structure:J}
For any $\varphi \in H^1$, we have
\begin{align*}
\left(\bar{\mu}-\LLL\right)E(\varphi) = & \int_{\R^3} \left(
\big|\nabla \varphi\big|^2  +
\big|\varphi\big|^6 \right)\; dx, \\
\LLL  \left(\bar{\mu}-\LLL\right)E(\varphi) = &
\int_{\R^3}\left(4\big|\nabla \varphi\big|^2 + 12\big|\varphi\big|^6
\right) \; dx.
\end{align*}
\end{lemma}
\begin{proof} By the definition of $\LLL$,
we have
\begin{align*}
\LLL \big\|\nabla \varphi\big\|^2_{L^2} = 4 \big\|\nabla
\varphi\big\|^2_{L^2},  \quad \LLL \big\|  \varphi\big\|^6_{L^6} =12 \big\|
\varphi\big\|^6_{L^6}, \quad  \LLL \big\|  \varphi\big\|^4_{L^4} =6
\big\| \varphi\big\|^4_{L^4},
\end{align*}
which implies that
\begin{align*}
\left(\bar{\mu}-\LLL\right)E(\varphi) = 6E(\varphi)-K(\varphi) = &
\int_{\R^3} \left( \big|\nabla \varphi\big|^2
+ \big|\varphi\big|^6 \right)\; dx,
\\
\LLL  \left(\bar{\mu}-\LLL\right)E(\varphi) =   \LLL \big\|\nabla
\varphi\big\|^2_{L^2} + \LLL
\big\|  \varphi\big\|^6_{L^6}  &= \int_{\R^3} \left( 4  \big|\nabla
\varphi\big|^2 + 12 \big| \varphi\big|^6 \right) \; dx.
\end{align*}
This completes the proof.
\end{proof}

According to the above analysis, we will replace the functional
$E$ in \eqref{minimization} with a positive functional $H$, while
extending the minimizing region from ``$K(\varphi)=0$'' to ``$K(\varphi)\leq 0$''. Let
\begin{align*}
H(\varphi):= \left(1 - \frac{\LLL}{\bar{\mu}}\right) E(\varphi)
=&\int_{\R^3} \left( \frac16 \big|\nabla \varphi\big|^2 +
\frac16 \big|\varphi\big|^6 \right)\;
dx,
\end{align*}
then for any $\varphi \in H^1 \backslash\{0\}$, we have
\begin{align*}
H(\varphi) > 0 , \quad  \LLL H(\varphi) \geq 0.
\end{align*}

Now we can characterization the minimization problem \eqref{minimization} by
use of $H$.
\begin{lemma}\label{minimization:H} For the minimization $m$ in \eqref{minimization}, we
have
\begin{align}
m  =& \inf  \{ H(\varphi)\; |\;  \varphi \in H^1(\R^3), \; \varphi\not =0,\;
K(\varphi) \leq 0  \} \nonumber\\
 =& \inf \{ H(\varphi)\; |\;\varphi \in H^1(\R^3), \; \varphi\not =0, \;
K(\varphi)< 0 \}. \label{JEqualH}
\end{align}
\end{lemma}
\begin{proof} For any $ \varphi\in H^1$, $\varphi \not=0$ with $K(\varphi)=0$, we have
$E(\varphi)=H(\varphi)$, this implies that
\begin{align}
 m= & \inf \{ E(\varphi)\; |\; \varphi \in H^1(\R^3),\; \varphi\not=0,\;
K(\varphi)=0 \}  \nonumber\\
\geq & \inf \{ H(\varphi)\; |\; \varphi \in
H^1(\R^3), \; \varphi\not=0, \; K(\varphi) \leq 0 \}.\label{JLargeH}
\end{align}

On the other hand, for any $\varphi \in H^1$, $\varphi\not=0$ with
$K(\varphi)<0$, by Lemma \ref{positive near origin:KQ }, Lemma \ref{Postivity:K}
and the continuity of $K$ in $\lambda$, we know that there
exists a $\lambda_0<0$ such that
\begin{align*}
K(\varphi^{\lambda_0}_{3,-2})=0,
\end{align*}
then by $\LLL H \geq 0$, we have
\begin{align*}
E(\varphi^{\lambda_0}_{3,-2}) = H (\varphi^{\lambda_0}_{3,-2}) \leq
H(\varphi^{0}_{3,-2})=H(\varphi).
\end{align*}
Therefore,
\begin{align}
&  \inf \{ E(\varphi)\; |\; \varphi \in H^1(\R^3), \; \varphi\not=0, \;
K(\varphi)=0 \}  \nonumber\\
 & \leq \inf \{ H(\varphi)\; |\;  \varphi \in
H^1(\R^3),\; \varphi\not=0, \; K(\varphi) < 0 \}.\label{JSmallH}
\end{align}
By \eqref{JLargeH} and \eqref{JSmallH}, we have
\begin{align*}
   & \inf  \{ H(\varphi)\; |\;  \varphi \in H^1(\R^3), \; \varphi\not =0,\;
K(\varphi) \leq 0  \} \\
 &\leq  m
\leq
  \inf \{ H(\varphi)\; |\;\varphi \in H^1(\R^3), \; \varphi\not =0, \;
K(\varphi)< 0 \}.
\end{align*}

In order to show \eqref{JEqualH}, it suffices to show that
\begin{align}
&\inf \{ H(\varphi)\; |\;  \varphi \in H^1(\R^3),\; \varphi\not=0, \; K(\varphi)
\leq 0 \} \nonumber\\
& \geq \inf \{ H(\varphi)\; |\; \varphi \in H^1(\R^3),\; \varphi\not=0,
\; K(\varphi)< 0 \}. \label{minimization:H:larger}
\end{align}
For any $\varphi \in H^1$, $\varphi\not=0$ with $K(\varphi)\leq 0$. By Lemma
\ref{structure:J}, we know that
\begin{align*}
\LLL K(\varphi)= \bar{\mu}K(\varphi)- \int_{\R^3}\left(4\big|\nabla
\varphi\big|^2 + 12\big|\varphi\big|^6 \right) \; dx <0,
\end{align*}
then for any $\lambda>0$ we have
\begin{align*}
K(\varphi^{\lambda}_{3,-2})<0,
\end{align*}
and as $\lambda\rightarrow 0$
\begin{align*}
H(\varphi^{\lambda}_{3,-2})=\int_{\R^3} \left(\frac{e^{4\lambda}}6
\big|\nabla \varphi \big|^2  + \frac{e^{12\lambda}}6  \big|\varphi\big|^6 \right) \; dx
\longrightarrow H(\varphi).
\end{align*}
This shows \eqref{minimization:H:larger}, and completes the proof.
\end{proof}

Next we will use the ($\dot H^1$-invariant) scaling argument to remove the $L^4$ term (the lower regularity quantity than $\dot H^1$) in $K$,
that is, to replace  the constrained
condition $K(\varphi) < 0$ with $K^c(\varphi) < 0$, where
\begin{align*}
K^{c}(\varphi):= \int_{\R^3} \left( 2 |\nabla \varphi|^2
-2|\varphi|^6 \right)\; dx.
\end{align*}

In fact, we have

\begin{lemma}\label{minimization:Hc}For the minimization $m$ in \eqref{minimization}, we
have
\begin{align*} m =& \inf \{ H(\varphi)\; |\; \varphi \in
H^1(\R^3),\; \varphi\not=0, \;
K^c(\varphi) < 0 \}\\
 =& \inf \{ H(\varphi)\; |\; \varphi \in   H^1(\R^3),\; \varphi\not =0, \;
K^c(\varphi) \leq  0 \}.
\end{align*}
\end{lemma}
\begin{proof}Since
$
K^c(\varphi) \leq K(\varphi)$, it is obvious that
\begin{align*}
m  = & \inf \{ H(\varphi)\; |\;  \varphi \in H^1(\R^3),\; \varphi\not=0, \;
K(\varphi)< 0 \}  \\
\geq & \inf \{ H(\varphi)\; |\; \varphi \in
H^1(\R^3),\; \varphi\not =0, \; K^c(\varphi) < 0 \}.
\end{align*}
Hence in order to show the first equality, it suffices to show that
\begin{align}
& \inf \{ H(\varphi)\; |\; \varphi \in H^1(\R^3),\; \varphi\not=0,  \; K(\varphi)<
0 \} \nonumber\\
&  \leq \inf \{ H(\varphi)\; |\;  \varphi \in   H^1(\R^3),\; \varphi\not=0,
\; K^c(\varphi) < 0 \}.\label{minimization:Hc:smaller}
\end{align}
To do so, for any $ \varphi \in H^1$, $\varphi\not=0$ with $ K^c(\varphi) < 0
$, taking
\begin{align*}
\varphi^{\lambda}_{1,-2}(x)=e^{\lambda}\varphi(e^{2\lambda}x),
\end{align*}
we have $  \varphi^{\lambda}_{1,-2} \in H^1$ and $ \varphi^{\lambda}_{1,-2}\not=0$ for any
$\lambda>0$. In addition,   we have
\begin{align*}
K(\varphi^{\lambda}_{1,-2})=  \int_{\R^3} \left( 2\big|\nabla \varphi\big|^2
- 2\big|\varphi\big|^6 + \frac32 e^{-2\lambda} \big|\varphi\big|^4 \right) \; dx & \longrightarrow K^c(\varphi),\\
H(\varphi^{\lambda}_{1,-2})= \int_{\R^3} \left( \frac16 \big|\nabla
\varphi\big|^2 + \frac16
\big|\varphi\big|^6 \right)\; dx = & H(\varphi),
\end{align*}
as $\lambda \rightarrow +\infty$. This gives
\eqref{minimization:Hc:smaller}, and completes the proof of the
first equality.

For the second equality, it is obvious that
\begin{align*}
& \inf \{ H(\varphi)\; |\; \varphi \in   H^1(\R^3),\; \varphi \not=0, \;
K^c(\varphi) < 0 \} \\
 &\geq \inf \{ H(\varphi)\; |\;  \varphi
\in H^1(\R^3),\; \varphi\not=0, \; K^c(\varphi) \leq  0 \},
\end{align*}
hence we only need to show that
\begin{align}
& \inf \{ H(\varphi)\; |\;  \varphi \in   H^1(\R^3),\; \varphi\not=0, \;
K^c(\varphi) < 0 \} \nonumber \\
 & \leq \inf \{ H(\varphi)\; |\;  \varphi
\in H^1(\R^3), \; \varphi\not=0, \; K^c(\varphi) \leq  0 \}.\label{minimization:Hc:smaller II}
\end{align}
To do this, we use the ($L^2$-invariant) scaling argument. For any $  \varphi \in H^1$, $\varphi\not=0$ with $K^c(\varphi)\leq
0$, we have $  \varphi^{\lambda}_{3,-2}\in H^1$, $\varphi^{\lambda}_{3,-2}\not=0$. In addition, by
\begin{align*}
\LLL K^c (\varphi) = & \int_{\R^3} \left(8\big|\nabla \varphi
\big|^2-24 \big|\varphi \big|^6 \right)\; dx = 4 K^{c}(\varphi)
-16\big\|\varphi\big\|^6_{L^6}<0,\\
& H(\varphi^{\lambda}_{3,-2}) =  \int_{\R^3} \left(
\frac{e^{4\lambda}}{6}
 \big|\nabla \varphi \big|^2 + \frac{e^{12\lambda}}{6}
\big|\varphi\big|^6\right) \; dx,
\end{align*}
we have $K^c(\varphi^{\lambda}_{3,-2})<0$ for any $\lambda>0$, and
\begin{align*}
H(\varphi^{\lambda}_{3,-2}) \rightarrow H(\varphi),\;\;
\text{as}\;\; \lambda \rightarrow 0.
\end{align*}
This implies \eqref{minimization:Hc:smaller II} and completes the
proof.
\end{proof}

After these preparations, we can now make use of the sharp Sobolev constant in \cite{Aubin:Sharp contant:Sobolev, Talenti:best constant}
to compute the minimization $m$ of \eqref{minimization}, which also shows Proposition \ref{threshold-energy}.

\begin{lemma}\label{threshold} For the minimization $m$ in \eqref{minimization}, we
have
\begin{align*}
m=E^c(W).
\end{align*}
\end{lemma}
\begin{proof} By Lemma \ref{minimization:Hc}, we have
\begin{align*}
m= & \inf \left\{\frac16 \int_{\R^3} \left(|\nabla \varphi|^2 +
|\varphi|^6\right)\; dx \; \Big| \;  \varphi \in H^1,\; \varphi\not=0,\;
\big\|\nabla \varphi \big\|^2_{L^2} \leq
\big\|\varphi\big\|^6_{L^6}   \right\}
\\
 \geq & \inf \left\{ \int_{\R^3} \frac16 \left(|\nabla \varphi|^2 +
|\varphi|^6\right) + \frac16 \left(|\nabla \varphi|^2-
|\varphi|^6\right)\; dx \; \Big|\;  \varphi \in H^1,\; \varphi\not=0,\;
\big\|\nabla \varphi \big\|^2_{L^2} \leq
\big\|\varphi\big\|^6_{L^6}  \right\}
\end{align*}
where the equality holds if and only if the minimization is taken
by some $\varphi$ with $\big\|\nabla \varphi \big\|^2_{L^2} =
\big\|\varphi\big\|^6_{L^6}$. While
\begin{align*}
& \inf\left\{ \int_{\R^3} \frac13  |\nabla \varphi|^2  \; dx \;
\big| \;  \varphi \in H^1,\;\varphi\not=0,\; \big\|\nabla \varphi
\big\|^2_{L^2} \leq
\big\|\varphi\big\|^6_{L^6}\right\} \\
&  = \inf\left\{ \frac13 \big\|\nabla
\varphi\big\|^2_{L^2} \left(
\frac{\big\|\nabla\varphi\big\|^2_{L^2}}{\big\|\varphi\big\|^6_{L^6}}\right)^{1/2}\;
\Big| \;   \varphi \in H^1, \; \varphi\not=0 \right\}\\
 & =\inf\left\{ \frac13   \left(
\frac{\big\|\nabla\varphi\big\|_{L^2}}{\big\|\varphi\big\|_{L^6}}\right)^{3}\;
\Big| \;  \varphi \in H^1, \; \varphi\not =0 \right\}\\
& = \inf\left\{ \frac13   \left(
\frac{\big\|\nabla\varphi\big\|_{L^2}}{\big\|\varphi\big\|_{L^6}}\right)^{3}\;
\Big| \;  \varphi \in \dot H^1,\; \varphi\not=0\right\}= \frac13
\big(C^*_3\big)^{-3}.
\end{align*}
where we use the density property $H^1\hookrightarrow \dot H^1$ in the last second equality
and that $C^*_3$ is the sharp Sobolev constant in $\R^3$, that is,
\begin{align*}
\big\|\varphi\big\|_{L^6_x}\leq C^*_3 \big\|\nabla \varphi
\big\|_{L^2_x},  \;\;\forall \; \varphi\in \dot H^1(\R^3),
\end{align*}
and the equality can be attained by the ground state $W$ of the following
elliptic equation \begin{align*} -\Delta W = |W|^4W.
\end{align*}
This implies that $\frac13 \big(C^*_3\big)^{-3}= E^c(W)$. The proof
is completed.
\end{proof}

After the computation of the minimization $m$ in
\eqref{minimization}, we next give some variational estimates.

\begin{lemma} \label{free-energ-equiva} For any $\varphi \in H^1$ with $K(\varphi)\geq 0$, we
have
\begin{align}\label{free energy}
\int_{\R^3} \left(\frac16\big|\nabla \varphi \big|^2  + \frac16 \big| \varphi\big|^6\right)  dx \leq
E(\varphi) \leq \int_{\R^3} \left(\frac12\big|\nabla \varphi \big|^2
  + \frac14 \big| \varphi\big|^4
\right) dx.
\end{align}
\end{lemma}
\begin{proof} On one hand, the right hand side of \eqref{free energy} is trivial.  On the other hand,
by the definition of $E$ and $K$, we have
\begin{align*}
E(\varphi)= \int_{\R^3} \left(\frac16\big|\nabla \varphi \big|^2 +
\frac16 \big| \varphi\big|^6\right)\;
dx + \frac{1}{6}K(\varphi),
\end{align*}
which implies the left hand side of \eqref{free energy}.
\end{proof}

At the last of this section, we give the uniform bounds on the
scaling derivative functional $K(\varphi)$ with the   energy
$E(\varphi)$ below the threshold $m$, which plays an important role for
the blow-up and scattering analysis in Section \ref{S:blow up} and Section \ref{S:GWP-Scattering}.

\begin{lemma}\label{uniform bound}
For any $\varphi \in H^1$ with $E(\varphi)<m$.
\begin{enumerate}
\item If $K(\varphi)<0$, then
\begin{align}\label{uniform:K:negative}
K(\varphi)  \leq -6\big(m-E(\varphi)\big).
\end{align}
\item If $K(\varphi)\geq 0$, then
\begin{align}\label{uniform:K:positive}
K(\varphi)\geq \min\left(6(m-E(\varphi)), \frac{2}{3} \big\|\nabla
\varphi \big\|^2_{L^2} +  \frac12 \big\|\varphi\big\|^4_{L^4}
\right).
\end{align}
\end{enumerate}
\end{lemma}
\begin{proof} By Lemma \ref{structure:J}, for any $\varphi \in H^1$, we have
\begin{align*}
\LLL^2 E(\varphi) = \bar{\mu} \LLL E(\varphi)- 4\big\|\nabla \varphi\big\|^2_{L^2} - 12
\big\|\varphi\big\|^6_{L^6}.
\end{align*}
Let $j(\lambda)=E(\varphi^{\lambda}_{3,-2})$, then we have
%\begin{align}\label{diff J} \fbox{\rule[-3mm]{0cm}{10mm}$
%j''(\lambda) =  \bar{\mu}j'(\lambda) - 4 e^{4\lambda}\big\|\nabla \varphi\big\|^2_{L^2} -
%12e^{12\lambda}\big\|\varphi\big\|^6_{L^6}.$}
%\end{align}
\begin{align}\label{diff J}
j''(\lambda) =  \bar{\mu}j'(\lambda) - 4 e^{4\lambda}\big\|\nabla \varphi\big\|^2_{L^2} -
12e^{12\lambda}\big\|\varphi\big\|^6_{L^6}.
\end{align}

\noindent{\bf Case I:} If $K(\varphi)<0$, then by
\eqref{asymptotic:KQ}, Lemma \ref{Postivity:K} and the continuity of $K$ in
$\lambda$, there exists a negative number
$\lambda_0<0$ such that $K(\varphi^{\lambda_0}_{3,-2})=  0$, and
\begin{align*}
K(\varphi^{\lambda}_{3,-2})<   0, \;\; \forall\; \; \lambda\in
(\lambda_0, 0).
\end{align*}
By \eqref{minimization}, we obtain $j(\lambda_0)=E(\varphi^{\lambda_0}_{3,-2})
\geq m$. Now by integrating \eqref{diff J} over
$[\lambda_0, 0]$, we have
\begin{align*}
\int^0_{\lambda_0} j''(\lambda)\; d\lambda \leq \bar{\mu}
\int^0_{\lambda_0} j'(\lambda)\; d\lambda,
\end{align*}
which implies that
\begin{align*}
K(\varphi)=j'(0)-j'(\lambda_0)\leq
\bar{\mu}\left(j(0)-j(\lambda_0)\right)  \leq -\bar{\mu}
\big(m-E(\varphi)\big),
\end{align*}
which implies \eqref{uniform:K:negative}.

\noindent{\bf Case II:} $K(\varphi) \geq 0$. We divide it
into two subcases:

When $2\bar{\mu} K(\varphi)\geq 12 \big\|\varphi\big\|^6_{L^6}$.
Since
\begin{align*}
12 \int_{\R^3}\big|\varphi\big|^6 \; dx = -6K(\varphi) + \int_{\R^3}
\left(12 \big| \nabla \varphi \big|^2 + 9 \big|  \varphi \big|^4
\right) \; dx,
\end{align*}
then we have
\begin{align*}
2\bar{\mu} K(\varphi)\geq -6K(\varphi) + \int_{\R^3} \left(12 \big|
\nabla \varphi \big|^2 + 9 \big|  \varphi \big|^4 \right) \; dx,
\end{align*}
which implies that
\begin{align*}
 K(\varphi)\geq \frac{2}{3} \big\|\nabla
\varphi \big\|^2_{L^2} +  \frac12 \big\|\varphi\big\|^4_{L^4}.
\end{align*}

When $2\bar{\mu}  K(\varphi) \leq 12 \big\|\varphi\big\|^6_{L^6}$.
By \eqref{diff J}, we have for $\lambda=0$
\begin{align}
0< &  2 \bar{\mu}j'(\lambda) <
12e^{12\lambda}\big\|\varphi\big\|^6_{L^6}, \nonumber\\
j''(\lambda) =
\bar{\mu}j'(\lambda) -& 4 e^{4\lambda}\big\|\nabla \varphi\big\|^2_{L^2}- 12e^{12\lambda}\big\|\varphi\big\|^6_{L^6}
\leq -\bar{\mu}j'(\lambda). \label{evolution j}
\end{align}
By the continuity of $j'$ and $j''$ in $\lambda$, we know that $j'$ is an accelerating decreasing function as $\lambda$ increases until $j'(\lambda_0)=0$ for some
finite number $\lambda_0>0$ and \eqref{evolution j} holds on $[0,
\lambda_0]$.

By
$
K(\varphi^{\lambda_0}_{3,-2})=j'(\lambda_0)=0,
$
we know that
\begin{align*}
E(\varphi^{\lambda_0}_{3,-2})\geq m.
\end{align*}
Now integrating \eqref{evolution j} over $[0, \lambda_0]$, we obtain that
\begin{align*}
-K(\varphi)=j'(\lambda_0)-j'(0) \leq -\bar{\mu} \big(j(\lambda_0)-j(0)\big)
\leq -\bar{\mu} (m-E(\varphi)).
\end{align*}
This completes the proof.
\end{proof}

%%%%%%%%%%%%%%%%%%%%%%%%%%%%%%%%%%%%%%%%%%%%%%%%%%%%%%%%%%%%%%%%%%%%%%%%%%%%%%%%%%%%%%%%%%%
%
%
%                                   Section three
%
%
%%%%%%%%%%%%%%%%%%%%%%%%%%%%%%%%%%%%%%%%%%%%%%%%%%%%%%%%%%%%%%%%%%%%%%%%%%%%%%%%%%%%%%%%%%%

\section{Part I: Blow up for $\KKK^-$}\label{S:blow up} In this section, we prove the
blow-up result of Theorem \ref{theorem}. We can also refer to \cite{OgaTsu:Blowup:NLS:91}.
Now let $\phi$ be a smooth, radial function satisfying $ \partial^2_r \phi(r) \leq 2$, $\phi(r)=r^2$ for $r\leq 1$, and $\phi(r)$ is constant for $r\geq 3$.
For some $R$, we define
\begin{align*}
V_R(t):=\int_{\R^3} \phi_R(x) |u(t,x)|^2\; dx, \quad
\phi_R(x)=R^2\phi\left(\frac{|x| }{R }\right).
\end{align*}

By Lemma \ref{L:virial},
$
\Delta\phi_R(r)=6 $  for $ r\leq R,$ and $\Delta^2 \phi_R(r)=0 $ for $  r\leq R,
$
we have
\begin{align*}
\partial^2_t V_R(t)
= &\; 4  \int_{\R^3} \phi_R''(r) \big|\nabla u(t,x)\big|^2 \;dx -\int_{\R^3} (\Delta^2 \phi_R)(x) |u(t,x)|^2 \; dx  \\
&  - \frac43 \int_{\R^3} (\Delta \phi_R)  |u(t,x)|^6\; dx +
\int_{\R^3} (\Delta \phi_R)  |u(t,x)|^4\; dx\\
%= & \; 4 \int_{\R^3} \left( 2 |\nabla u (t)|^2 -2|u |^6 +\frac32
%|u (t)|^4 \right)\; dx -\int_{\R^3}\Delta^2 \phi_R |u(t)|^2 \; dx \\
%& + \int_{|x|\geq R}\left(\big(4\phi_R''-8\big)|\nabla u (t) |^2 + \big(8-\frac43\Delta\phi_R\big) \big| u(t)  \big|^6
%+ \big(\Delta\phi_R-6\big)  \big| u (t) \big|^4
% \right)\; dx \\
\leq &\; 4 \int_{\R^3} \left( 2 |\nabla u (t)|^2 -2|u (t)|^6 +\frac32
|u (t)|^4 \right)\; dx \\
 & +  \frac{c}{R^2}\int_{R\leq |x|\leq 3R} \big| u (t) \big|^2 \; dx + c \int_{R\leq |x|\leq 3R} \left( \big| u (t) \big|^4
 + \big| u(t) \big|^6 \right) \; dx.
\end{align*}

By the Gagliardo-Nirenberg  and radial Sobolev inequalities, we have
\begin{align*}
\big\|f\big\|^4_{L^4(|x|\geq R)} \leq & \frac{c}{R^2} \big\|f\big\|^3_{L^2(|x|\geq R)}\big\|\nabla f\big\|_{L^2(|x|\geq R)},\\
\big\|f\big\|_{L^\infty(|x|\geq R)} \leq & \frac{c}{R} \big\|f\big\|^{1/2}_{L^2(|x|\geq R)}\big\|\nabla f\big\|^{1/2}_{L^2(|x|\geq R)}.
\end{align*}
Therefore, by mass conservation and Young's inequality, we know that for any $\epsilon>0$ there exist sufficiently large $R$ such that
\begin{align}
\partial^2_t V_R(t) \leq & 4 K(u(t))
+ \epsilon \big\|\nabla u(t,x)\big\|^2_{L^2 } + \epsilon^2. \nonumber\\
= & 48 E(u) - \big(16-\epsilon\big)\big\|\nabla u(t)\big\|^2_{L^2} - 6\big\|u(t)\big\|^4_{L^4} + \epsilon^2\label{videntity:second der}
\end{align}
By $K(u)<0$, mass and energy conservations, Lemma \ref{uniform bound} and
the continuity argument, we know that for any $t\in I$, we have
\begin{align*}
K(u(t)) \leq -6\left(m-E(u(t))\right)<0.
\end{align*}
By Lemma \ref{minimization:H}, we have
\begin{align*}
m \leq H(u(t))< \frac13 \big\|u(t)\big\|^6_{L^6}.
\end{align*}
where we have used the fact that $K(u(t))<0$ in the second inequality. By the fact $m=\frac13 \left(C^*_3\right)^{-3}$ and the Sharp Sobolev inequality, we have
\begin{align*}
\big\|\nabla u(t)\big\|^6_{L^2} \geq \left(C^*_3\right)^{-6} \big\|u(t)\big\|^6_{L^6} > \left(C^*_3\right)^{-9},
\end{align*}
which implies that $\big\|\nabla u(t)\big\|^2_{L^2} > 3m$.

In addition, by $E(u_0)<m$ and energy conservation, there exists $\delta_1>0$ such that
$E(u(t))\leq (1-\delta_1)m$. Thus, if we choose $\epsilon$ sufficiently small, we have
\begin{align*}
\partial^2_t V_R(t) \leq 48 (1-\delta_1)m - 3\big(16-\epsilon\big) m +   \epsilon^2 \leq -24 \delta_1 m,
\end{align*}
which implies that $u$ must blow up at finite time. \qed

%%%%%%%%%%%%%%%%%%%%%%%%%%%%%%%%%%%%%%%%%%%%%%%%%%%%%%%%%%%%%%%%%%%%%%%%%%%%%%%%%%%%%%%%%%%
%
%
%                                   Section four
%
%
%%%%%%%%%%%%%%%%%%%%%%%%%%%%%%%%%%%%%%%%%%%%%%%%%%%%%%%%%%%%%%%%%%%%%%%%%%%%%%%%%%%%%%%%%%%

\section{Perturbation theory}\label{S:long time perturbation}
In this part, we give the perturbation theory of the
solution of \eqref{NLS} with the global space-time estimate. First we
denote the space-time space $ST(I)$ on the time interval $I$ by
\begin{align*}
ST(I):=   \left(L^{10}_t \dot B^{1/3}_{90/19, 2} \cap L^{12}_tL^9_x \cap L^6_t\dot B^{1/2}_{18/7,2}\cap
L^{5}_{t,x}\right)&(I\times \R^3),\\
ST^*(I):=   \left( L^2_t \dot B^{1/3}_{18/11,2}  \cap
L^2_t\dot B^{1/2}_{6/5,2}\right)(I\times \R^3)&.
\end{align*}

The main result in this section is the following.
\begin{proposition}\label{stability}
Let $I$ be a compact time interval and let $w$ be an approximate
solution to \eqref{NLS} on $I\times \R^3$ in the sense that
\begin{align*}
i\partial_t w + \Delta w= - |w|^{4} w+|w|^2 w +e
\end{align*}
for some suitable small function $e$. Assume that for some constants $L, E_0>0$, we
have
\begin{align*}
\big\|w\big\|_{ST(I)} \leq  L, \quad  \big\|w (t_0)\big\|_{
H^1_x(\R^3)}  \leq  E_0
\end{align*}
for some $t_0\in I$.  Let $u(t_0)$ close
to $w(t_0)$ in the sense that for some $E'>0$, we have
\begin{align*}
\big\|u(t_0)-w(t_0)\big\|_{ H^1_x}\leq E'.
\end{align*}
 Assume also that for some $\varepsilon$, we have
\begin{align} \label{small:gamma0}
  \left\|  e^{i(t-t_0)\Delta}\big(u(t_0)-w(t_0)\big) \right\|_{ST(I)} & \leq   \varepsilon,\quad \big\|  e\big\|_{ST^*(I)} \leq   \varepsilon,
\end{align}
where $  0<\varepsilon \leq
\varepsilon_0=\varepsilon_0(E_0, E', L) $ is a small constant.
Then there exists a solution $u$ to \eqref{NLS} on $I\times \R^3$
with initial data $u(t_0)$ at time $t=t_0$ satisfying
\begin{align*}
\big\| u-w \big\|_{ST(I)} \leq& C(E_0, E', L) \;
\varepsilon, \quad \text{and}\quad  \big\|u\big\|_{  ST (I)} \leq
C(E_0, E', L).
\end{align*}
\end{proposition}

\begin{proof} Since $w\in ST(I)$, there exists a partition of the
right half of $I$ at $t_0$:
\begin{align*}
t_0<t_1<\cdots<t_N,\quad I_j=(t_j, t_{j+1}), \quad I\cap (t_0,
\infty)=(t_0, t_N),
\end{align*}
such that $N\leq C(L, \delta)$ and for any $j=0, 1, \ldots, N-1$, we
have
\begin{align}\label{small:w}
\big\|w\big\|_{ST(I_j)}\leq \delta\ll 1 .
\end{align}
The estimate on the left half of $I$ at $t_0$ is analogue, we omit it.

Let
\begin{align*}
\gamma(t,x)= & u(t,x)-w(t,x), \\
\gamma_j(t,x)=& e^{i(t-t_j)\Delta}\Big(u(t_j,x)-w(t_j,x)\Big),
\end{align*}
then $\gamma$ satisfies the following difference equation
\begin{align*}
i\gamma_t + \Delta \gamma = O(w^4\gamma + w^3\gamma^2 + w^2 \gamma^3
+ w\gamma^4 + \gamma^5 + w^2\gamma + w \gamma^2 + \gamma^3) -e,
\end{align*}
which implies that
\begin{align*}
\gamma(t)= & \gamma_j(t) -i \int^t_{t_j}e^{i(t-s)\Delta}
\Big(O(w^4\gamma + w^3\gamma^2 + w^2 \gamma^3 + w\gamma^4 + \gamma^5
+ w^2\gamma + w \gamma^2 + \gamma^3) -e\Big) \; ds,
\\
\gamma_{j+1}(t)= & \gamma_j(t) -i
\int^{t_{j+1}}_{t_j}e^{i(t-s)\Delta} \Big(O(w^4\gamma + w^3\gamma^2
+ w^2 \gamma^3 + w\gamma^4 + \gamma^5 + w^2\gamma + w \gamma^2 +
\gamma^3) -e\Big) \; ds.
\end{align*}

By Lemma \ref{Strichartz }, we have
\begin{align}
& \big\|\gamma-\gamma_j\big\|_{L^6_t\left(I_j; \dot B^{1/2}_{18/7,2}\right)\cap
L^{5}_{t,x}\left(I_j\right)} +
\big\|\gamma_{j+1}-\gamma_j\big\|_{L^6_t\left(\R; \dot B^{1/2}_{18/7,2}\right)\cap
L^{5}_{t,x}\left(\R \times \R^3 \right)} \label{estimate:gamma:S0}\\
& \lesssim  \big\| O(w^4\gamma + w^3\gamma^2 + w^2 \gamma^3 +
w\gamma^4 + \gamma^5  \big\|_{L^2_t\left(I_j; \dot B^{1/2}_{6/5,2}\right)} \nonumber \\
&\quad  + \big\| w^2\gamma + w \gamma^2 + \gamma^3)
\big\|_{L^2_t\left(I_j; \dot B^{1/2}_{6/5,2}\right)} +
\big\|e\big\|_{L^2_t\left(I_j; \dot B^{1/2}_{6/5,2}\right)} \nonumber \\
& \lesssim
\big\|w\big\|^4_{L^{12}_t\left(I_j;L^{9}_x\right)}\big\|\gamma\big\|_{L^6_t\left(I_j; \dot B^{1/2}_{18/7,2}\right) }
+ \big\|w\big\|^3_{L^{12}_t\left(I_j;L^{9}_x\right)} \big\|w\big\|_{L^6_t\left(I_j; \dot B^{1/2}_{18/7,2}\right) }\big\|\gamma\big\|_{L^{12}_t\left(I_j;L^{9}_x\right)}\nonumber \\
& \quad +
 \big\|w\big\|^3_{L^{12}_t\left(I_j;L^{9}_x\right)}
 \big\|\gamma\big\|_{L^{12}_t\left(I_j;L^{9}_x\right)}\big\|\gamma\big\|_{L^6_t\left(I_j; \dot B^{1/2}_{18/7,2}\right)} +
  \big\|w\big\|^2_{L^{12}_t\left(I_j;L^{9}_x\right)}\big\|w\big\|_{L^6_t\left(I_j; \dot B^{1/2}_{18/7,2}\right)}
 \big\|\gamma\big\|^2_{L^{12}_t\left(I_j;L^{9}_x\right)}\nonumber
 \\
 & \quad +
 \big\|w\big\|^2_{L^{12}_t\left(I_j;L^{9}_x\right)}
 \big\|\gamma\big\|^2_{L^{12}_t\left(I_j;L^{9}_x\right)}\big\|\gamma\big\|_{L^6_t\left(I_j; \dot B^{1/2}_{18/7,2}\right)} +
  \big\|w\big\|_{L^{12}_t\left(I_j;L^{9}_x\right)}\big\|w\big\|_{L^6_t\left(I_j; \dot B^{1/2}_{18/7,2}\right)}
 \big\|\gamma\big\|^3_{L^{12}_t\left(I_j;L^{9}_x\right)}\nonumber
 \\
 & \quad +
 \big\|w\big\|_{L^{12}_t\left(I_j;L^{9}_x\right)}
 \big\|\gamma\big\|^3_{L^{12}_t\left(I_j;L^{9}_x\right)}\big\|\gamma\big\|_{L^6_t\left(I_j; \dot B^{1/2}_{18/7,2}\right)} +
\big\|w\big\|_{L^6_t\left(I_j; \dot B^{1/2}_{18/7,2}\right)}
 \big\|\gamma\big\|^4_{L^{12}_t\left(I_j;L^{9}_x\right)}\nonumber \\
& \quad + \big\|\gamma\big\|^4_{L^{12}_t\left(I_j;L^{9}_x\right)}\big\|\gamma\big\|_{L^6_t\left(I_j; \dot B^{1/2}_{18/7,2}\right)} \nonumber\\
 & \quad +
 \big\|w\big\|^2_{L^6\left(I_j; L^{9/2}_x\right)}\big\|\gamma\big\|_{L^6_t\left(I_j; \dot B^{1/2}_{18/7,2}\right)} +
 \big\|w\big\| _{L^6\left(I_j; L^{9/2}_x\right)}\big\|w\big\|_{L^6_t\left(I_j; \dot B^{1/2}_{18/7,2}\right)} \big\|\gamma\big\|_{L^6\left(I_j; L^{9/2}_x\right)}\nonumber
 \\
 & \quad +
 \big\|w\big\|_{L^6\left(I_j; L^{9/2}_x\right)} \big\|\gamma\big\|_{L^6\left(I_j; L^{9/2}_x\right)}\big\|\gamma\big\|_{L^6_t\left(I_j; \dot B^{1/2}_{18/7,2}\right)} +
\big\|w\big\|_{L^6_t\left(I_j; \dot B^{1/2}_{18/7,2}\right)} \big\|\gamma\big\|^2_{L^6\left(I_j; L^{9/2}_x\right)}\nonumber
 \\
 & \quad +  \big\|\gamma\big\|^2_{L^6\left(I_j; L^{9/2}_x\right)}\big\|\gamma\big\|_{L^6_t\left(I_j; \dot B^{1/2}_{18/7,2}\right)}
  + \big\|e\big\|_{L^2_t\left(I_j; \dot B^{1/2}_{6/5,2}\right)}. \nonumber
\end{align}

At the same time, by Lemma \ref{strichartz:exotic}, we have
\begin{align}
& \big\|\gamma-\gamma_j\big\|_{L^{10}_t\left(I_j; \dot B^{1/3}_{90/19, 2}\right) \cap L^{12}_t\left(I_j; L^9_x \right)}
+ \big\|\gamma_{j+1}-\gamma_j\big\|_{L^{10}_t \left(\R; \dot B^{1/3}_{90/19, 2} \right)\cap L^{12}_t\left(\R; L^9_x\right) } \label{estimate:gamma:S1}\\
\lesssim & \left\| O(w^4\gamma + w^3\gamma^2 + w^2 \gamma^3 +
w\gamma^4 + \gamma^5 + w^2\gamma + w \gamma^2 + \gamma^3)
\right\|_{L^2(I_j; \dot B^{1/3}_{\frac{18}{11},2})} +   \big\| e
\big\|_{L^2(I_j; \dot B^{1/3}_{\frac{18}{11},2})}  \nonumber \\
\lesssim &
\big\|w\big\|^4_{L^{10}_{t,x}(I_j)}\big\|\gamma\big\|_{L^{10}_t(I_j;\dot
B^{1/3}_{90/19, 2})}
+\big\|w\big\|^3_{L^{10}_{t,x}(I_j)}\big\|w\big\|_{L^{10}_t(I_j;\dot
B^{1/3}_{90/19, 2})}\big\|\gamma\big\|_{L^{10}_{t,x}(I_j)}  \nonumber\\
& +
\big\|w\big\|^3_{L^{10}_{t,x}(I_j)}\big\|\gamma\big\|_{L^{10}_{t,x}(I_j)}\big\|\gamma\big\|_{L^{10}_t(I_j;\dot
B^{1/3}_{90/19, 2})} +
\big\|w\big\|^2_{L^{10}_{t,x}(I_j)}\big\|w\big\|_{L^{10}_t(I_j;\dot
B^{1/3}_{90/19, 2})}\big\|\gamma\big\|^2_{L^{10}_{t,x}(I_j)}  \nonumber\\
& +
\big\|w\big\|^2_{L^{10}_{t,x}(I_j)}\big\|\gamma\big\|^2_{L^{10}_{t,x}(I_j)}\big\|\gamma\big\|_{L^{10}_t(I_j;\dot
B^{1/3}_{90/19, 2})} +
\big\|w\big\|_{L^{10}_{t,x}(I_j)}\big\|w\big\|_{L^{10}_t(I_j;\dot
B^{1/3}_{90/19, 2})}\big\|\gamma\big\|^3_{L^{10}_{t,x}(I_j)}   \nonumber\\
& +
\big\|w\big\|_{L^{10}_{t,x}(I_j)}\big\|\gamma\big\|^3_{L^{10}_{t,x}(I_j)}\big\|\gamma\big\|_{L^{10}_t(I_j;\dot
B^{1/3}_{90/19, 2})} + \big\|w\big\|_{L^{10}_t(I_j;\dot
B^{1/3}_{90/19, 2})}\big\|\gamma\big\|^4_{L^{10}_{t,x}(I_j)} \nonumber \\
&  +
\big\|\gamma\big\|^4_{L^{10}_{t,x}(I_j)}\big\|\gamma\big\|_{L^{10}_t(I_j;\dot
B^{1/3}_{90/19, 2})}  \nonumber\\
& +
\big\|w\big\|^2_{L^5_{t,x}(I_j)}\big\|\gamma\big\|_{L^{10}_t(I_j;\dot
B^{1/3}_{90/19, 2})} +
\big\|w\big\|_{L^5_{t,x}(I_j)}\big\|w\big\|_{L^{10}_t(I_j;\dot
B^{1/3}_{90/19, 2})} \big\|\gamma\big\|_{L^5_{t,x}(I_j)}  \nonumber\\
&+
\big\|w\big\|_{L^5_{t,x}(I_j)}\big\|\gamma\big\|_{L^5_{t,x}(I_j)}\big\|\gamma\big\|_{L^{10}_t(I_j;\dot
B^{1/3}_{90/19, 2})} + \big\|w\big\|_{L^{10}_t(I_j;\dot
B^{1/3}_{90/19, 2})}
 \big\|\gamma\big\|^2_{L^5_{t,x}(I_j)} \nonumber\\
 &  +
 \big\|\gamma\big\|^2_{L^5_{t,x}(I_j)} \big\|\gamma\big\|_{L^{10}_t(I_j;\dot
B^{1/3}_{90/19, 2})}   +   \big\| e \big\|_{L^2(I_j; \dot
B^{1/3}_{18//11,2})}. \nonumber
\end{align}

By the interpolation, we have
\begin{align*}
\big\|f\big\|_{L^6\left(I_j; L^{9/2}_x\right)}
\lesssim\big\|f\big\|_{L^6_t\left(I_j; \dot B^{1/2}_{18/7,2}\right)}, \quad
\big\|f\big\|_{L^{10}_{t,x}(I_j)} \lesssim
\big\|f\big\|_{L^{10}_t(I_j;\dot B^{1/3}_{90/19, 2})}.
\end{align*}
Therefore, assuming that
\begin{align} \label{small:gamma}
\big\|\gamma\big\|_{ST(I_j)} \leq \delta \ll 1,\quad \forall \; j=0,
1, \ldots, N-1,
\end{align}
then by \eqref{small:w}, \eqref{estimate:gamma:S0} and
\eqref{estimate:gamma:S1}, we have
\begin{align*}
\big\|\gamma\big\|_{ST(I_j)} + \big\|\gamma_{j+1}\big\|_{ST(t_{j+1},
t_N)} \leq C\big\|\gamma_j\big\|_{ST(t_j, t_N)} + \varepsilon ,
\end{align*}
for some absolute constant $C>0$. By \eqref{small:gamma0} and iteration on $j$, we get
\begin{align*}
\big\|\gamma\big\|_{ST(I)} \leq (2C)^N \varepsilon \leq
\frac{\delta}{2},
\end{align*}
if we choose $\varepsilon_0$ sufficiently small. Hence the
assumption \eqref{small:gamma} is justified by continuity in $t$ and
induction on $j$. then repeating the estimate
\eqref{estimate:gamma:S0} and \eqref{estimate:gamma:S1} once again,
we can obtain the $ST$-norm estimate on $\gamma$, which implies
the Strichartz estimate on $u$.
\end{proof}

%%%%%%%%%%%%%%%%%%%%%%%%%%%%%%%%%%%%%%%%%%%%%%%%%%%%%%%%%%%%%%%%%%%%%%%%%%%%%%%%%%%%%%%%%%%
%
%
%                                   Section five
%
%
%%%%%%%%%%%%%%%%%%%%%%%%%%%%%%%%%%%%%%%%%%%%%%%%%%%%%%%%%%%%%%%%%%%%%%%%%%%%%%%%%%%%%%%%%%%

\section{Profile decomposition} In this part, we will use the method in \cite{BahG:NLW:proffile decomp, IbrMN:f:NLKG, Ker:NLS:profile decomp}
to show the linear and nonlinear profile
decompositions of the sequences of radial, $H^1$-bounded solutions of \eqref{NLS}, which will be used to construct the critical
element (minimal  energy non-scattering solution) and show its properties, especially the compactness. In order to do it, we now introduce the complex-valued function
$\overrightarrow{v}(t,x)$ by
\begin{align*}
\overrightarrow{v}(t,x)=\left<\nabla\right>v (t,x), \quad v (t,x)
=\left<\nabla\right>^{-1}\overrightarrow{v} (t,x).
\end{align*}

Given $(t^j_n, h^j_n)\in \R \times (0, 1]$, let
$\tau^j_n$, $T^j_n$ denote the scaled time drift, the  scaling transformation, defined by
\begin{align*}
\tau^j_n = - \frac{t^j_n}{\left(h^j_n\right)^2}, \quad T^j_n \varphi(x) =
\frac{1}{(h^j_n)^{3/2}} \varphi \left(\frac{x}{h^j_n}\right).
\end{align*}

We also introduce the set of Fourier multipliers on
$\R^3$.
\begin{align*} \mathcal{MC} = \{\mu =
\mathcal{F}^{-1}\widetilde{\mu}\mathcal{F}\; | \; \widetilde{\mu}\in
C(\R^3), \; \exists \lim_{|\xi|\rightarrow +\infty}
\widetilde{\mu}(\xi)\in \R \}.
\end{align*}

\subsection{Linear profile decomposition}In this subsection, we show
the profile decomposition with the scaling parameter of a sequence of the radial, free Schr\"{o}dinger
solutions in the energy space $H^1(\R^3)$, which implies the profile decomposition of a sequence of radial initial data.

\begin{proposition}\label{L:linear profile}
Let
$$\overrightarrow{v}_n(t,x)=e^{it\Delta}\overrightarrow{v}_n(0)$$
 be a sequence of the  radial solutions of the free
Schr\"{o}dinger equation with bounded $L^2$ norm. Then up to a
subsequence, there exist $K\in \{0, 1,2,\ldots, \infty\}$, radial functions
$\{\varphi^j\}_{j\in [0, K)}\subset L^2(\R^3)$ and $\{t^j_n,
h^j_n\}_{n\in \N} \subset \R  \times (0, 1]$ satisfying
\begin{align}\label{profile:linear}
\overrightarrow{v}_n(t,x) = \sum^{k-1}_{j=0}
\overrightarrow{v}^j_n(t,x) + \overrightarrow{w}^k_n(t,x),
\end{align}
where $ \overrightarrow{v}^j_n (t,x) = e^{i(t-t^j_n)\Delta} T^j_n
\varphi^j$, and
\begin{align}\label{small:w weak topology}
\lim_{k\rightarrow K} \varlimsup_{n\rightarrow +\infty}
\big\|\overrightarrow{w}^k_n\big\|_{L^{\infty}_t(\R;
B^{-3/2}_{\infty, \infty}(\R^3))} =0,
\end{align}
and for any Fourier multiplier $\mu \in \mathcal{MC}$, any
$l<j<k\leq K$ and any $t\in \R$,
\begin{align}
\lim_{n\rightarrow +\infty} \left(\log \left| \dfrac{h^j_n}{h^l_n} \right|
+ \left| \frac{t^j_n - t^l_n}{(h^l_n)^2} \right|  \right)=\infty,\label{orth:I}
\end{align}
\begin{align}
\lim_{n\rightarrow +\infty} \left< \mu \overrightarrow{v}^l_n(t)
\;,\; \mu \overrightarrow{v}^j_n(t) \right>_{L^2_x} =   \lim_{n\rightarrow +\infty}  \left< \mu
\overrightarrow{v}^j_n(t) \;,\; \mu \overrightarrow{w}^k_n(t)
\right>_{L^2_x} = 0 . \label{orth:II}
\end{align}
Moreover, each sequence $\{h^j_n\}_{n\in \N}$ is either going to $0$
or identically $1$ for all $n$.
\end{proposition}
\begin{remark}\label{R:free energy orthogonality} We call $\overrightarrow{v}^j_n$ and $\overrightarrow{w}^k_n$ the free concentrating wave
and the remainder, respectively. From \eqref{orth:II}, we have the following
asymptotic orthogonality
\begin{align}\label{orth:II:linear energy}
%\lim_{n\rightarrow +\infty} \left( \big\|\overrightarrow{v}_n(t)
%\big\|^2_{L^2}- \sum^{k-1}_{j=0}
% \big\|\overrightarrow{v}^j_n(t)
%\big\|^2_{L^2} -  \big\|\overrightarrow{w}^k_n(t)
%\big\|^2_{L^2}\right)= & 0,\\
\lim_{n\rightarrow +\infty} \left( \big\|\mu \overrightarrow{v}_n(t)
\big\|^2_{L^2}- \sum^{k-1}_{j=0}
 \big\|\mu \overrightarrow{v}^j_n(t)
\big\|^2_{L^2} -  \big\|\mu \overrightarrow{w}^k_n(t)
\big\|^2_{L^2}\right) & =0.
\end{align}
\end{remark}

\begin{proof}[Proof of Proposition \ref{L:linear profile}] Let
\begin{align*}
\nu :=\varlimsup_{n\rightarrow
\infty}\big\|\overrightarrow{v}_n\big\|_{L^{\infty}_tB^{-3/2}_{\infty,
\infty}} = \varlimsup_{n\rightarrow \infty} \sup_{(t,x)\in \R\times
\R^3,\atop k\geq 0} 2^{-3k/2}\big| \Lambda_k *
\overrightarrow{v}_n(t,x) \big|.
\end{align*}

If $\nu=0$, then we have done with $K=0$.

Otherwise, $\displaystyle \nu=\varlimsup_{n\rightarrow
\infty}\big\|\overrightarrow{v}_n\big\|_{L^{\infty}_tB^{-3/2}_{\infty,
\infty}}>0$. By the radial Gagliardo-Nirenberg inequality and the Bernstein inequality, we have
\begin{align*}
\sup_{t\in \R, |2^k x|\geq R,\atop k\geq 0} 2^{-3k/2}\big| \Lambda_k *
\overrightarrow{v}_n(t,x) \big| \lesssim & \sup_{ k\geq 0} \frac{2^k 2^{-3k/2}}{R}\big\|\Lambda_k *
\overrightarrow{v}_n(t,x)\big\|^{1/2}_{L^{\infty}_tL^2_x}\cdot \big\|\nabla \Lambda_k *
\overrightarrow{v}_n(t,x)\big\|^{1/2}_{L^{\infty}_tL^2_x}\\
\lesssim & \sup_{ k\geq 0} \frac{1}{R}\big\|
\overrightarrow{v}_n(t,x)\big\| _{L^{\infty}_tL^2_x} \lesssim \frac1R.
\end{align*}
If taking $R$ sufficiently large, we have
\begin{align*}  \sup_{t\in \R, |2^k x|\geq R, k\geq 0} 2^{-3k/2}\big| \Lambda_k *
\overrightarrow{v}_n(t,x) \big|\leq \frac12 \nu.
\end{align*}
thus, there exists a
sequence $(t_n, x_n, k_n)$ with $k_n \geq 0$ and $|2^{k_n}x_n|\leq R$ such that for large
$n$,
\begin{align*}
\frac12 \varlimsup_{n\rightarrow
 \infty}\big\|\overrightarrow{v}_n\big\|_{L^{\infty}_tB^{-3/2}_{\infty,
\infty}}  = \frac12 \nu \leq 2^{-3k_n/2}\big| \Lambda_{k_n} *
\overrightarrow{v}_n(t_n,x_n) \big|.
\end{align*}

Now we define $h_n$ and $\psi_n$ by $h_n = 2^{-k_n} \in (0, 1]$ and
\begin{align}
\overrightarrow{v}_n(t_n, x) = & \left(T_n \psi_n\right)(x-x_n)= \frac{1}{(h_n)^{3/2}}
\psi_n \left(\frac{x-x_n}{h_n}\right)\label{decom:I}\\
= & T_n \left(\psi_n\Big(x-\frac{x_n}{h_n}\Big)\right).\nonumber
\end{align}
Since $ \big\|\psi_n\big\|_{L^2} = \big\| T_n \psi_n\big\|_{L^2} =
\big\|\overrightarrow{v}_n(t_n)\big\|_{L^2} \leq C, $ then there
exists some $\psi \in L^2$, such that, up to a subsequence, we have as $n\rightarrow +\infty$
\begin{align}\label{decom:I:prop I}
\frac{x_n}{h_n} \rightarrow x^0, \;\; \text{and}\;\; \psi_n\rightharpoonup \psi\quad\text{weakly in}\;\;  L^2.
\end{align}

On the other hand, if $k_n =0$, we have
\begin{align*}
2^{-3k_n/2}\big| \Lambda_{k_n} * \overrightarrow{v}_n(t_n,x_n) \big|
= & \displaystyle \int_{\R^3} \Lambda_{0}(y)\;\;
2^{-3k_n/2}\overrightarrow{v}_n\left(t_n, x_n -
\frac{y}{2^{k_n}}\right) \; dy  \\
= & \int_{\R^3} \Lambda_{0}(y)\;
\psi_n(-y) \; dy  \\
 \longrightarrow & \int_{\R^3} \Lambda_{0}(y)\;
\psi(-y) \; dy  \lesssim   \big\|\psi\big\|_{L^2}.
\end{align*}
By the same way, if $k_n \geq 1$, we have
\begin{align*}
2^{-3k_n/2}\big| \Lambda_{k_n} * \overrightarrow{v}_n(t_n,x_n) \big|
= & \int_{\R^3} \Lambda_{(0)}(y)\;
2^{-3k_n/2}\overrightarrow{v}_n\left(t_n, x_n -
\frac{y}{2^{k_n}}\right) \; dy \\
= &  \int_{\R^3} \Lambda_{(0)}(y)\;\psi_n(-y) \; dy \\
 \longrightarrow &
\int_{\R^3} \Lambda_{(0)}(y)\;\psi(-y) \; dy  \lesssim
\big\|\psi\big\|_{L^2}.
\end{align*}
If $h_n\rightarrow 0$, then we take
\begin{align*}
(t^0_n,  h^0_n)=  (t_n,   h_n), \quad \varphi^0(x)=\psi\left(x-x^0\right),
\end{align*}
otherwise, up to a subsequence, we may assume that $h_n\rightarrow h_{\infty}$ for some
$h_{\infty} \in (0, 1]$, and take
\begin{align*}
(t^0_n,  h^0_n)=  (t_n,   1), \quad
\varphi^0(x)=\frac{1}{(h_{\infty})^{3/2}}\psi\left(\frac{x}{h_{\infty}}-x^0\right),
\end{align*}
then
\begin{align}\label{decom:I:prop II}
T_n\left(\psi \Big(x-\frac{x_n}{h_n}\Big)\right)- T^0_n\varphi^0 (x)\longrightarrow 0 \quad\text{strongly in}\;\; L^2.
\end{align}
In addition, since $\overrightarrow{v}_n(t_n, x) = \left(T_n \psi_n\right)(x-x_n)$ is radial, so is $\varphi^0(x)$.

Let $\overrightarrow{v}^0_n (t,x) =  e^{i(t-t^0_n)\Delta}T^0_n
\varphi^0 $,  we define $\overrightarrow{w}^1_n$ by
\begin{align}\label{decom:I:profile}
\overrightarrow{v}_n (t,x) = & \overrightarrow{v}^0_n (t,x) +
\overrightarrow{w}^1_n (t,x),
\end{align}
then by \eqref{decom:I:prop I} and \eqref{decom:I:prop II}, we have
\begin{align*}
(T^0_n)^{-1}\overrightarrow{w}^1_n(t^0_n) = (T^0_n)^{-1}T_n \left(\psi_n\Big(x-\frac{x_n}{h_n}\Big)\right)
-\varphi^0 \rightharpoonup 0\quad \text{weakly in}\;\; L^2,
\end{align*}
which implies that
\begin{align*}
\left<\mu \overrightarrow{v}^0_n(t), \mu
\overrightarrow{w}^1_n(t)\right> = & \left<\mu
\overrightarrow{v}^0_n(t^0_n), \mu
\overrightarrow{w}^1_n(t^0_n)\right>  =  \left<\mu^0_n \varphi^0,
\mu^0_n (T^0_n)^{-1} \overrightarrow{w}^1_n(t^0_n)\right>
\longrightarrow 0,
\end{align*}
where we used the conservation law in the first equality and the dominated convergence theorem and
$\mu^0_n(D)=\mu\left(\frac{D}{h^0_n}\right)$ in the last equality. It is the
decomposition for $k=1$.

Next we apply the above procedure to the sequence
$\overrightarrow{w}^1_n$ in place of $\overrightarrow{v}_n$, then
either $\displaystyle \varlimsup_{n\rightarrow
\infty}\big\|\overrightarrow{w}^1_n\big\|_{L^{\infty}_tB^{-3/2}_{\infty,
\infty}} =0$ or we can find the next concentrating wave
$\overrightarrow{v}^1_n$ and the remainder $\overrightarrow{w}^2_n$,
such that for some $(t^1_n,  h^1_n)$ with $h^1_n \in (0, 1]$
and radial function $\varphi^1 \in L^2(\R^3)$,
\begin{align}\label{decom:II}
\overrightarrow{w}^1_n(t,x)= \overrightarrow{v}^1_n(t,x) + &
\overrightarrow{w}^2_n(t,x)  = e^{i(t-t^1_n)\Delta}T^1_n\varphi^1(x)
+ \overrightarrow{w}^2_n (t,x),\end{align}
and
\begin{align}\label{decom:II:prop}
\varlimsup_{n\rightarrow +\infty}
\big\|\overrightarrow{w}^1_n\big\|_{L^{\infty}_tB^{-3/2}_{\infty,
\infty}} \lesssim  &\;
\big\|\varphi^1\big\|_{L^2} = \big\|\overrightarrow{v}^1_n\big\|_{L^2},
\\
(T^1_n)^{-1}\overrightarrow{w}^2_n(t^1_n) \rightharpoonup   0\quad
\text{weakly in}\;\; L^2 & \Longrightarrow   \left<\mu
\overrightarrow{v}^1_n(t), \mu \overrightarrow{w}^2_n(t)\right>
\longrightarrow
  0. \nonumber
\end{align}

Iterating the above procedure, we can obtain the decomposition
\eqref{profile:linear}. It remains to show the properties \eqref{small:w weak topology}, \eqref{orth:I} and \eqref{orth:II}.

We first assume that  \eqref{orth:II} holds, then by
\eqref{orth:II:linear energy} and the Cauchy criterion, we have
\begin{align}\label{decom:k:prop}
\lim_{n\rightarrow +\infty}
\big\|\overrightarrow{w}^k_n\big\|_{L^{\infty}_tB^{-3/2}_{\infty,
\infty}} \lesssim \big\|\varphi^k\big\|_{L^2} =
\big\|\overrightarrow{v}^k_n\big\|_{L^2}\longrightarrow
  0\quad \text{as}\;\; k\rightarrow +\infty.
\end{align}
which implies \eqref{small:w weak topology}.

Now we show \eqref{orth:I} by contradiction. Suppose that \eqref{orth:I} fails,
then there exists a minimal $(l,j)$ which violates \eqref{orth:I}. By
extracting a subsequence, We may assume that $h^l_n \rightarrow
h^l_{\infty}$ and $ h^l_n/h^j_n $ and $(t^l_n - t^j_n)/(h^l_n)^2$   all converge.

Now consider
\begin{align*}
\left(T^l_n\right)^{-1}\overrightarrow{w}^{l+1}_n(t^l_n) = &
\sum^{j}_{m=l+1}\left(T^l_n\right)^{-1}\overrightarrow{v}^m_n(t^l_n)
+
\left(T^l_n\right)^{-1}\overrightarrow{w}^{j+1}_n(t^l_n) \\
= & \sum^{j}_{m=l+1}\left(T^l_n\right)^{-1}e^{i(t^l_n -t^m_n)\Delta}
T^m_n\varphi^m +
\left(T^l_n\right)^{-1}\overrightarrow{w}^{j+1}_n(t^l_n) \\
=&  \sum^{j-1}_{m=l+1} S^{l,m}_n\varphi^m + S^{l,j}_n\varphi^j +
\left(T^l_n\right)^{-1}\overrightarrow{w}^{j+1}_n(t^l_n),
\end{align*}
where
\begin{align*}
S^{l,m}_n=\left(T^l_n\right)^{-1}e^{i(t^l_n -t^m_n)\Delta} T^m_n =
e^{i\frac{t^l_n -t^m_n}{(h^l_n)^2}\Delta}
\left(T^l_n\right)^{-1}T^m_n:=e^{it^{l,m}_n\Delta} T^{l,m}_n
\end{align*}
with the sequence
\begin{align}\label{combine:parameters}
t^{l,m}_n= \frac{t^l_n -t^m_n}{(h^l_n)^2}, \quad h^{l,m}_n = \frac{h^m_n}{h^l_n}.
\end{align}

By the procedure of constructing \eqref{profile:linear}, as $n\rightarrow +\infty$, we have
\begin{align*}
\left(T^l_n\right)^{-1}\overrightarrow{w}^{l+1}_n(t^l_n)
\rightharpoonup    0 & \quad \text{weakly in}\; L^2,\\
\left(T^j_n\right)^{-1}\overrightarrow{w}^{j+1}_n(t^j_n)
\rightharpoonup   0 & \quad \text{weakly in}\; L^2,
 \end{align*}and by the asymptotic
 orthogonality \eqref{orth:I} between $m$ and $l$ with $m \in [l+1,
j-1]$
\begin{align*}
S^{l,m}_n \varphi^m \rightharpoonup   0, \;\; \forall\; m \in [l+1,
j-1],
\end{align*}
and by the convergence of $ h^l_n/h^j_n $ and   $(t^l_n -
t^j_n)/(h^l_n)^2$, we have $ S^{l,j}_n\varphi^j \rightarrow
S^{l,j}_{\infty}\varphi^j$ and
\begin{align*}
\left(T^l_n\right)^{-1}\overrightarrow{w}^{j+1}_n(t^l_n) = &
S^{l,j}_n\left(T^j_n\right)^{-1}\overrightarrow{w}^{j+1}_n(t^j_n)
\rightharpoonup   0 \quad  \text{weakly in}\; L^2.
\end{align*}
Then $\varphi^j=0$, it is a contradiction. Thus we obtain the
orthogonality \eqref{orth:I}.

Last we show \eqref{orth:II}. For $j\not = l$, we have
\begin{align*}
& \left< \mu \overrightarrow{v}^l_n(t) \;,\; \mu
\overrightarrow{v}^j_n(t) \right>_{L^2_x}\\
 = &  \left< \mu
\overrightarrow{v}^l_n(0) \;,\; \mu \overrightarrow{v}^j_n(0)
\right>_{L^2_x}  =   \left< \mu e^{-it^l_n\Delta}T^l_n \varphi^l
\;,\; \mu
e^{-it^j_n\Delta}T^j_n\varphi^j\right>_{L^2_x}\\
= & \left< e^{-it^l_n\Delta}T^l_n \mu^l_n \varphi^l \;, \;
e^{-it^j_n\Delta}T^j_n \mu^j_n \varphi^j\right>_{L^2_x}  =  \left<
\left(T^j_n\right)^{-1} e^{i(t^j_n-t^l_n)\Delta}T^l_n \mu^l_n
\varphi^l \;,\;
 \mu^j_n \varphi^j\right>_{L^2_x} \\
 = & \left< e^{i\frac{t^j_n-t^l_n}{(h^j_n)^2}\Delta} \left(T^j_n\right)^{-1}T^l_n
\mu^l_n \varphi^l \;,\;
 \mu^j_n \varphi^j\right>_{L^2_x}
 =  \left< S^{j,l}_n\mu^l_n \varphi^l \;,\;
 \mu^j_n \varphi^j\right>_{L^2_x}
  \rightarrow   0 \quad \text{as}\;\;
 n\rightarrow +\infty
\end{align*}
where $\widetilde{\mu}^l_n (\xi)= \widetilde{\mu}
\left(\xi/h^l_n\right)$ and we used the fact that $S^{j,l}_n
\rightharpoonup 0$ weakly in $L^2$ as $n\rightarrow +\infty$ by
\eqref{orth:I}. In addition, we have
\begin{align*}
\left< \mu \overrightarrow{v}^j_n(t) \;, \; \mu
\overrightarrow{w}^k_n(t) \right>_{L^2_x} = \left< \mu
\overrightarrow{v}^j_n(t) \;,\; \mu \Big(
\overrightarrow{w}^{j+1}_n(t)-\sum^{k-1}_{m=j+1}\overrightarrow{v}^{m}_n(t)\Big)
\right>_{L^2_x}\longrightarrow 0
\end{align*}
as $n \rightarrow +\infty$. This completes the proof of
\eqref{orth:II}.
\end{proof}

After the orthogonality's proof of the linear energy, we begin with the orthogonal analysis for the nonlinear energy.

\begin{lemma}Let $\overrightarrow{v}_n$ be a sequence of the radial solutions of the free Schr\"{o}dinger equation. Let
\begin{align*}
\overrightarrow{v}_n(t,x) = \sum^{k-1}_{j=0}
\overrightarrow{v}^j_n(t,x) + \overrightarrow{w}^k_n(t,x)
\end{align*} be the linear profile decomposition given by Proposition \ref{L:linear profile}. Then we have
\begin{align*}
\lim_{k\rightarrow K}  \varlimsup_{n\rightarrow +\infty} \left|
M(v_n(0)) - \sum^{k-1}_{j=0} M(v^j_n(0)) -M(w^k_n(0))\right| =0, \\
\lim_{k\rightarrow K}  \varlimsup_{n\rightarrow +\infty} \left|
E(v_n(0)) - \sum^{k-1}_{j=0}E(v^j_n(0)) - E (w^k_n(0))\right| =0, \\
\lim_{k\rightarrow K} \varlimsup_{n\rightarrow +\infty} \left|
K(v_n(0)) - \sum^{k-1}_{j=0} K(v^j_n(0)) -K(w^k_n(0))\right| =0.
\end{align*}

\end{lemma}
\begin{proof} We can show that the quadratic terms in $M$, $E$ and $K$ have the orthogonal decomposition
by taking $ \mu=\frac{1}{\left<\nabla\right>} $ and $ \mu=\frac{|\nabla|}{\left<\nabla\right>} $
in Remark \ref{R:free energy orthogonality}, thus it suffices to show that
\begin{align*}
\lim_{k\rightarrow K} \varlimsup_{n\rightarrow +\infty} \left|
F_i\left(v_n(0)\right) - \sum_{j<k} F_i\left(v^j_n(0)\right) - F_i\left(w^k_n(0)\right) \right| =0,
\quad i=1,2,
\end{align*}
where $F_1$ and $F_2$ are denoted by
\begin{align*}
F_1(u(t)) =   \int_{\R^3}
|u(t,x)|^6  \; dx, \;\; F_2(u(t))=  \int_{\R^3}
 |u(t,x)|^4   \; dx.
\end{align*}

In order to do so, we need re-arrange the linear concentrating wave with respect to its dispersive decay
(whether $\tau^j_n$ goes to $\pm \infty$ or not for all $j$).
Let $v^{<k}_n(0)= \displaystyle \sum_{j<k}v^j_n(0)= \sum_{j<k,
\tau^j_n \rightarrow \tau^j_{\infty}}v^j_n(0) +  \sum_{j<k, \tau^j_n
\rightarrow \pm \infty} v^j_n(0)$ for some finite numbers
$\tau^j_{\infty}$'s, then we have
\begin{align}
\Big| F_i\left(v_n(0)\right) -&  \sum_{j<k}
F_i\left(v^j_n(0)\right) - F_i\left(w^k_n(0)\right)
\Big| \nonumber\\
 \leq & \left| F_i\left(v_n(0)\right) - F_i\left(v^{<k}_n(0)\right) \right| +
\left|F_i\left(v^{<k}_n(0)\right) - F_i\left(\sum_{j<k, \tau^j_n
\rightarrow
\tau^j_{\infty}} v^j_n(0) \right)\right| \nonumber\\
& + \left|F_i\left(\sum_{j<k, \tau^j_n \rightarrow
\tau^j_{\infty}}v^j_n(0) \right) -  \sum_{j<k, \tau^j_n \rightarrow
\tau^j_{\infty}} F_i\left(v^j_n(0)\right) \right| \label{orth:non dispersive}\\
&  + \left| \sum_{j<k, \tau^j_n \rightarrow \pm \infty}
F_i\left(v^j_n(0)\right) \right| +\left| F_i\left(w^k_n(0)\right)
\right|. \nonumber
\end{align}

First, by \eqref{small:w weak topology} and interpolation, we have
that
\begin{align*}
\lim_{k\rightarrow K} \varlimsup_{n\rightarrow +\infty}
\big\|w^k_n(0)\big\|_{L^p_x}=0,\quad \forall \; \; 2 < p \leq 6.
\end{align*}
which implies that
\begin{align*}
\lim_{k\rightarrow K} \varlimsup_{n\rightarrow +\infty} \left|
F_i\left(v_n(0)\right) - F_i\left(v^{<k}_n(0)\right) \right|& =0,\\
\lim_{k\rightarrow K} \varlimsup_{n\rightarrow +\infty} \left|
F_i\left(w^k_n(0)\right) \right| =0.&
\end{align*}

Second by the dispersive estimate for $v^j_n(0)$ with $\tau^j_n
\rightarrow \pm \infty$, we have
\begin{align*}
\lim_{k\rightarrow K} \varlimsup_{n\rightarrow +\infty}
\left|F_i(v^{<k}_n(0)) - F_i\left(\sum_{j<k, \tau^j_n \rightarrow
\tau^j_{\infty}} v^j_n(0) \right)\right| & =0, \\
\lim_{k\rightarrow K} \varlimsup_{n\rightarrow +\infty} \left|
\sum_{j<k, \tau^j_n \rightarrow \pm \infty} F_i\left(v^j_n(0)\right)
\right|=0. &
\end{align*}

Last we will use the approximation argument in \cite{IbrMN:f:NLKG} to show that every non-dispersive concentrating wave
will get away from the others, which contributes to the orthogonality of \eqref{orth:non dispersive}.
Let $\psi^j := e^{i\tau^j_{\infty} \Delta}\varphi^j \in L^2$,
we have
\begin{align}
& \left|F_i\left(\sum_{j<k, \tau^j_n \rightarrow \tau^j_{\infty}}v^j_n(0)
\right) -  \sum_{j<k, \tau^j_n \rightarrow \tau^j_{\infty}}
F_i\left(v^j_n(0)\right)
\right|\\
&\leq  \left|F_i\left(\sum_{j<k, \tau^j_n \rightarrow
\tau^j_{\infty}}v^j_n(0) \right) - F_i\left(\sum_{j<k, \tau^j_n
\rightarrow \tau^j_{\infty}}
\left<\nabla\right>^{-1} T^j_n  \psi^j  \right) \right| \nonumber\\
&\quad  + \left| \sum_{j<k, \tau^j_n \rightarrow \tau^j_{\infty}}
F_i\left(v^j_n(0)\right) - \sum_{j<k, \tau^j_n \rightarrow
\tau^j_{\infty}}
F_i\left(\left<\nabla\right>^{-1} T^j_n  \psi^j )\right) \right| \nonumber\\
&\quad + \left| F_i\left(\sum_{j<k, \tau^j_n \rightarrow \tau^j_{\infty}}
\left<\nabla\right>^{-1} T^j_n  \psi^j  \right) - \sum_{j<k,
\tau^j_n \rightarrow \tau^j_{\infty}}
F_i\left(\left<\nabla\right>^{-1} T^j_n \psi^j )\right)  \right|.
\label{orth:nonlinearity:most difficult}
\end{align}
For those $v^j_n(0)$ with $\tau^j_n \rightarrow \tau^j_{\infty}$, by
the continuity of the operator $e^{it\Delta}$ in $t$ in $H^1$, we
have
\begin{align*}
v^j_n(0)=& \left<\nabla\right>^{-1} e^{-it^j_n\Delta}T^j_n\varphi^j
=  \left<\nabla\right>^{-1} T^j_n
e^{i\tau^j_n\Delta}\varphi^j\longrightarrow \left<\nabla\right>^{-1}
T^j_n  \psi^j \quad \text{in}\;\; H^1(\R^3),
\end{align*}
which implies that
\begin{align*}
\left|F_i\left(\sum_{j<k, \tau^j_n \rightarrow
\tau^j_{\infty}}v^j_n(0) \right) - F_i\left(\sum_{j<k, \tau^j_n
\rightarrow \tau^j_{\infty}}
\left<\nabla\right>^{-1} T^j_n  \psi^j  \right) \right| \rightarrow 0,\\
 \left| \sum_{j<k, \tau^j_n \rightarrow \tau^j_{\infty}}
F_i\left(v^j_n(0)\right) - \sum_{j<k, \tau^j_n \rightarrow
\tau^j_{\infty}} F_i\left(\left<\nabla\right>^{-1} T^j_n  \psi^j
)\right) \right|\rightarrow 0.
\end{align*}

Now we consider \eqref{orth:nonlinearity:most difficult} for $i=1,
2$, separately.

First for $i=2$, we compute as following,
\begin{align*}
& \left| F_2\left(\sum_{j<k, \tau^j_n \rightarrow \tau^j_{\infty}}
\left<\nabla\right>^{-1} T^j_n  \psi^j  \right) - \sum_{j<k,
\tau^j_n \rightarrow \tau^j_{\infty}}
F_2\left(\left<\nabla\right>^{-1} T^j_n
\psi^j )\right)  \right| \\
 & \leq \left| F_2\left(\sum_{j<k, \tau^j_n \rightarrow
\tau^j_{\infty}} \left<\nabla\right>^{-1} T^j_n  \psi^j  \right)  -
F_2\left(\sum_{j<k, \tau^j_n \rightarrow \tau^j_{\infty}, h^j_n=1}
\left<\nabla\right>^{-1} T^j_n  \psi^j  \right) \right| \\
&\quad + \left| \sum_{j<k, \tau^j_n \rightarrow \tau^j_{\infty}}
F_2\left(\left<\nabla\right>^{-1} T^j_n \psi^j )\right)  -\sum_{j<k,
\tau^j_n \rightarrow \tau^j_{\infty}, h^j_n=1}
F_2\left(\left<\nabla\right>^{-1} T^j_n \psi^j )\right)\right| \\
& \quad + \left| F_2\left(\sum_{j<k, \tau^j_n \rightarrow \tau^j_{\infty},
h^j_n=1} \left<\nabla\right>^{-1} T^j_n  \psi^j  \right) -
\sum_{j<k, \tau^j_n \rightarrow \tau^j_{\infty}, h^j_n=1}
F_2\left(\left<\nabla\right>^{-1} T^j_n \psi^j )\right)\right|.
\end{align*}

For $h^j_n \rightarrow 0$, we have
\begin{align*}
\left<\nabla\right>^{-1} T^j_n  \psi^j  \rightarrow 0\quad
\text{in}\;\; L^p,\;\; \forall\; 2\leq p<6,
\end{align*}
which implies that
\begin{align*}
\left| F_2\left(\sum_{j<k, \tau^j_n \rightarrow \tau^j}
\left<\nabla\right>^{-1} T^j_n  \psi^j  \right)  -
F_2\left(\sum_{j<k, \tau^j_n \rightarrow \tau^j, h^j_n=1}
\left<\nabla\right>^{-1} T^j_n  \psi^j  \right) \right| \rightarrow 0, \\
\left| \sum_{j<k, \tau^jn_n \rightarrow \tau^j}
F_2\left(\left<\nabla\right>^{-1} T^j_n \psi^j )\right)  -\sum_{j<k,
\tau^j_n \rightarrow \tau^j, h^j_n=1}
F_2\left(\left<\nabla\right>^{-1} T^j_n \psi^j )\right)\right|
\rightarrow 0.
\end{align*}
In addition, by the orthogonality \eqref{orth:I}, we know that there is at most one term $\left<\nabla\right>^{-1} T^j_n \psi^j$
with $\tau^j_n \rightarrow
\tau^j_{\infty}, h^j_n= 1$, hence
\begin{align*}
\left| F_2\left(\sum_{j<k, \tau^j_n \rightarrow \tau^j_{\infty},
h^j_n=1} \left<\nabla\right>^{-1} T^j_n  \psi^j  \right) -
\sum_{j<k, \tau^j_n \rightarrow \tau^j_{\infty}, h^j_n=1}
F_2\left(\left<\nabla\right>^{-1} T^j_n \psi^j )\right)\right|
=0.
\end{align*}

Now we consider the case $i=1$, Let $\widehat{\psi}^j =
|\nabla|^{-1}\psi^j$ if $h^j_n \rightarrow 0 $, and
$\widehat{\psi}^j =\left<\nabla\right>^{-1}\psi^j$ if $h^j_n \equiv
1$, then we have $\widehat{\psi}^j \in L^6_x$, and
\begin{align*}
& \left| F_1\left(\sum_{j<k, \tau^j_n \rightarrow \tau^j_{\infty}}
\left<\nabla\right>^{-1} T^j_n  \psi^j  \right) - \sum_{j<k,
\tau^j_n \rightarrow \tau^j_{\infty}}
F_1\left(\left<\nabla\right>^{-1} T^j_n
\psi^j )\right)  \right| \\
 & \leq\left| F_1\left(\sum_{j<k, \tau^j_n \rightarrow
\tau^j_{\infty}} \left<\nabla\right>^{-1} T^j_n  \psi^j  \right)
 -  F_1\left(\sum_{j<k, \tau^j_n \rightarrow \tau^j_{\infty}} h^j_n T^j_n \widehat{\psi}^j\right)\right|\\
&\quad +\left| \sum_{j<k, \tau^j_n \rightarrow \tau^j_{\infty}}
F_1\left(\left<\nabla\right>^{-1} T^j_n \psi^j )\right) - \sum_{j<k,
\tau^j_n \rightarrow \tau^j_{\infty}} F_1 \left( h^j_n
T^j_n\widehat{ \psi}^j\right) \right| \\
&\quad + \left| F_1\left(\sum_{j<k, \tau^j_n \rightarrow \tau^j_{\infty}}
h^j_n T^j_n \widehat{\psi}^j\right)- \sum_{j<k, \tau^j_n \rightarrow
\tau^j_{ \infty}} F_1 \left( h^j_n T^j_n \widehat{\psi}^j\right)\right|.
\end{align*}

Since
\begin{align*}
  \big\|\left<\nabla\right>^{-1}T^j_n\psi^j - h^j_n T^j_n
\widehat{\psi}^j\big\|_{L^6_x}  = &   \begin{cases}
\big\|\left<\nabla\right>^{-1}T^j_n\psi^j - h^j_n T^j_n
|\nabla|^{-1}\psi^j \big\|_{L^6_x}  \quad \text{if}\;\; h^j_n \rightarrow 0\\
\big\|\left<\nabla\right>^{-1}T^j_n\psi^j -  h^j_nT^j_n
\left<\nabla\right>^{-1}\psi^j \big\|_{L^6_x}  \quad \text{if}\;\;
h^j_n \equiv 1
\end{cases}  \\
 = &   \begin{cases}
\big\|(\left<\nabla\right>^j_n)^{-1} \psi^j -
|\nabla|^{-1}\psi^j \big\|_{L^6_x}  \quad \text{if}\;\; h^j_n \rightarrow 0\\
0 \quad \text{if}\;\;
h^j_n \equiv 1
\end{cases}  \\
 &  \longrightarrow 0, \quad \text{as}\;\; n\rightarrow +\infty,
\end{align*}
which shows that
\begin{align*}
\left| F_1\left(\sum_{j<k, \tau^j_n \rightarrow \tau^j_{\infty}}
\left<\nabla\right>^{-1} T^j_n  \psi^j  \right)
 -  F_1\left(\sum_{j<k, \tau^j_n \rightarrow \tau^j_{\infty}} h^j_n T^j_n \widehat{\psi}^j\right)\right|\longrightarrow 0, \\
 \left| \sum_{j<k, \tau^j_n \rightarrow \tau^j_{\infty}}
F_1\left(\left<\nabla\right>^{-1} T^j_n \psi^j )\right) - \sum_{j<k,
\tau^j_n \rightarrow \tau^j_{\infty}} F_1 \left( h^j_n
T^j_n\widehat{ \psi}^j\right) \right| \longrightarrow 0.
\end{align*}
We further replace each $\widehat{\psi}^j$ by the non-overlap terms $\widetilde{\psi}^j_n$ with each other
\begin{align*}
\widetilde{\psi}^j_n = \widehat{\psi}^j \times
\begin{cases} 0; \quad
\exists\; l<j, \;\text{such that}\; h^l_n < h^j_n \; \; \text{and}\;\;
\frac{x }{h^{j,l}_n} \in \text{supp}\; {\widehat{\psi}^l}, \\
1; \quad \text{otherwise},
\end{cases}
\end{align*}
where $  h^{j,l}_n$ is determined by
\eqref{combine:parameters}. By \eqref{orth:I}, we know that $h^{j,l}_n \rightarrow 0$,
therefore as $n\rightarrow +\infty$ \begin{align*}
\widetilde{\psi}^j_n \rightarrow \widehat{\psi}^j, & \quad a. e.\;
x\in \R^3, \quad \text{and}\quad   \widetilde{\psi}^j_n \rightarrow
\widehat{\psi}^j,
 \quad \text{in}\;\; L^{6}_x,
\end{align*}
which implies that
\begin{align*}
 \left| F_1\left(\sum_{j<k, \tau^j_n \rightarrow \tau^j_{\infty}}
h^j_n T^j_n \widehat{\psi}^j\right) - F_1\left(\sum_{j<k, \tau^j_n
\rightarrow \tau^j_{\infty}} h^j_n T^j_n \widetilde{\psi}^j_n\right)\right|\longrightarrow 0, \\
 \left| \sum_{j<k, \tau^j_n \rightarrow \tau^j_{\infty}} F_1 \left( h^j_n
T^j_n \widehat{\psi}^j\right) - \sum_{j<k, \tau^j_n \rightarrow
\tau^j_{\infty}} F_1 \left( h^j_n T^j_n \widetilde{\psi}^j_n\right)
\right| \longrightarrow 0.
\end{align*}
On the other hand, by the support property of
$\widetilde{\psi}^j_n$, we know that
\begin{align*}
  F_1\left(\sum_{j<k, \tau^j_n \rightarrow \tau^j_{\infty}} h^j_n T^j_n
\widetilde{\psi}^j_n\right) = \sum_{j<k, \tau^j_n \rightarrow
\tau^j_{\infty}} F_1 \left( h^j_n T^j_n \widetilde{\psi}^j_n\right).
\end{align*}
Therefore, we have
\begin{align*}
& \left| F_1\left(\sum_{j<k, \tau^j_n \rightarrow \tau^j_{\infty}}
h^j_n T^j_n \widehat{\psi}^j\right)- \sum_{j<k, \tau^j_n \rightarrow
\tau^j_{\infty}} F_1
\left( h^j_n T^j_n \widehat{\psi}^j\right)\right| \\
 &\leq \left| F_1\left(\sum_{j<k, \tau^j_n \rightarrow
\tau^j_{\infty}} h^j_n T^j_n \widehat{\psi}^j\right) -
F_1\left(\sum_{j<k, \tau^j_n
\rightarrow \tau^j_{\infty}} h^j_n T^j_n \widetilde{\psi}^j_n\right)\right|\\
&\quad + \left| \sum_{j<k, \tau^j_n \rightarrow \tau^j_{\infty}} F_1
\left( h^j_n T^j_n \widehat{\psi}^j\right) - \sum_{j<k, \tau^j_n
\rightarrow \tau^j_{\infty}} F_1 \left( h^j_n T^j_n
\widetilde{\psi}^j_n\right) \right|
  \longrightarrow 0.
\end{align*}
This completes the proof.
\end{proof}

\begin{lemma}\label{L:close of K}
Let $k\in \N$ and radial functions $\varphi_0, \ldots, \varphi_k \in H^1(\R^3)$, $m$ be determined by \eqref{minimization}.
Assume that there exist some $\delta$, $\varepsilon>0$ with $ 4
\varepsilon  < 3 \delta$ such that
\begin{align*}
\sum^k_{j=0}E(\varphi_j) - \varepsilon \leq
E\left(\sum^k_{j=0}\varphi_j\right) < m-\delta,\;\; \text{and}\;\;
 -\varepsilon \leq
K\left(\sum^k_{j=0}\varphi_j\right) \leq
\sum^k_{j=0}K(\varphi_j)+\varepsilon.
\end{align*}
Then $\varphi_j \in \KKK^+$ for all $j=0, \ldots, k$.
\end{lemma}

\begin{proof}Suppose that $K(\varphi_l) < 0 $ for some $l$. Then by Lemma \ref{minimization:H}, we have
\begin{align*}
H(\varphi_l) \geq \inf\left\{H(\varphi) \; | \;   \varphi \in
H^1(\R^3),\; \varphi \not=0,\;  K(\varphi) \leq 0\right\} = m.
\end{align*}
By the nonnegativity of $H(\varphi_j)$ for $j\geq 0$, we have
\begin{align*}
m \leq&  H(\varphi_l) \leq \sum^{k}_{j=0} H(\varphi_j)  =
\sum^{k}_{j=0}\left( E(\varphi_j)- \frac16 K(\varphi_j) \right)\\
\leq & E\left(\sum^k_{j=0}\varphi_j\right) + \varepsilon -\frac16
K\left(\sum^k_{j=0}\varphi_j\right)  + \frac16 \varepsilon \\
\leq & m -\delta + \varepsilon +  \frac13 \varepsilon < m.
\end{align*}
It is a contradiction. Hence for any $j\in \{0, \ldots, k\}$, we
have
\begin{align*}
K(\varphi_j) \geq 0,
\end{align*}
which implies that
\begin{align*}
E(\varphi_j) = H(\varphi_j)+ \frac16 K(\varphi_j) \geq 0,
\end{align*}
and
\begin{align*}
E(\varphi_j) \leq \sum^k_{i=0}E(\varphi_i) <m-\delta + \varepsilon <
m,
\end{align*}
which means that $\varphi_j \in \KKK^+$ for all $j$.
\end{proof}

According to the above results, we conclude as following.
\begin{proposition}\label{decomp:stable:K}
Let $\overrightarrow{v}_n(t,x) $ be a sequence of the radial solutions of the free Schr\"{o}dinger equation satisfying
\begin{align*}
v_n(0)\in \KKK^+ \;\; \text{and}\;\;  E (v_n(0))<m.
\end{align*}

Let
\begin{align*}
\overrightarrow{v}_n(t,x) = \sum^{k-1}_{j=0}
\overrightarrow{v}^j_n(t,x) + \overrightarrow{w}^k_n(t,x),
\end{align*}
be the linear profile decomposition given by Proposition
\ref{L:linear profile}. Then for large $n$ and all $j<K$, we have
\begin{align*}
v^j_n(0)\in \KKK^+,\;\;  \;\; w^K_n(0) \in \KKK^+,
\end{align*}
 and
\begin{align*}
\lim_{k\rightarrow K} \varlimsup_{n\rightarrow +\infty} \left|
M(v_n(0)) - \sum_{j<k} M(v^j_n(0)) - M(w^k_n(0))\right| =0, \\
\lim_{k\rightarrow K} \varlimsup_{n\rightarrow +\infty} \left|
E(v_n(0)) - \sum_{j<k}E(v^j_n(0)) - E(w^k_n(0))\right| =0, \\
\lim_{k\rightarrow K} \varlimsup_{n\rightarrow +\infty} \left|
K(v_n(0)) - \sum_{j<k} K(v^j_n(0)) -K(w^k_n(0))\right| =0.
\end{align*}
Moreover for all $j<K$, we have
\begin{align*}
0 \leq  \varliminf_{ n\rightarrow +\infty} E(v^j_n(0)) \leq
\varlimsup_{n\rightarrow +\infty} E(v^j_n(0)) \leq
\varlimsup_{n\rightarrow +\infty} E(v_n(0)),
\end{align*}
where the last inequality becomes equality only if $K=1$ and $w^1_n
\rightarrow 0$ in $L^{\infty}_t \dot H^1_x$.
\end{proposition}

\subsection{Nonlinear profile decomposition} After the linear profile decomposition of a sequence of initial data in
the last subsection, we now show the nonlinear profile decomposition of a sequence
of radial solutions of \eqref{NLS} with the same initial data in the energy
space $H^1(\R^3)$. First we introduce some notation
\begin{align*}
\left<\nabla\right>^j_n = \sqrt{\left(h^j_n\right)^2 - \Delta}, \;\; \left<\nabla\right>^j_{\infty} = \sqrt{\left(h^j_{\infty}\right)^2 - \Delta}\; .
\end{align*}

Now let $v_n(t,x)$ be a sequence of radial solutions for the free
Schr\"{o}dinger equation with initial data in $\KKK^+$, that is,
$v_n \in H^1(\R^3)$ is radial and
\begin{align*}
\left(i\partial_t + \Delta\right) v_n =0,\quad v_n(0)\in \KKK^+.
\end{align*}

Let
\begin{align*}
\overrightarrow{v}_n(t,x) = \left<\nabla\right> v_n(t,x),
\end{align*}
then by Proposition \ref{L:linear profile}, we have a sequence of the radial,
free concentrating wave $\overrightarrow{v}^j_n(t,x)$ with
$\overrightarrow{v}^j_n (t^j_n) = T^j_n \varphi^j$, $v^j_n(0)\in
\KKK^+$ for $j=0, \ldots, K$, such that
\begin{align*}
\overrightarrow{v}_n(t,x) = & \sum^{k-1}_{j=0}
\overrightarrow{v}^j_n(t,x) + \overrightarrow{w}^k_n(t,x)  =
\sum^{k-1}_{j=0} e^{i(t-t^j_n)\Delta}T^j_n \varphi^j
  + \overrightarrow{w}^k_n \\
  = &  \sum^{k-1}_{j=0} T^j_n e^{i\left(\frac{t-t^j_n}{(h^j_n)^2}\right)\Delta} \varphi^j
  + \overrightarrow{w}^k_n .
\end{align*}

Now for any concentrating wave $\overrightarrow{v}^j_n $, $j=0,
\ldots, K$, we undo the group action, i.e., the scaling  transformation $T^j_n$,  to look for the
linear profile $V^j$. Let
\begin{align*}
\overrightarrow{v}^j_n(t,x) = &  T^j_n
\overrightarrow{V}^j\left(\frac{t-t^j_n}{(h^j_n)^2}\right),
\end{align*}
then we have
\begin{align*}
 \left(i\partial_t + \Delta\right) \overrightarrow{V}^j =0, \quad
  \overrightarrow{V}^j(0)=\varphi^j.
\end{align*}

Now let $u^j_n(t,x)$ be the nonlinear solution of \eqref{NLS} with
initial data $v^j_n(0)$, that is
\begin{align*}
 \left(i\partial_t + \Delta\right) \overrightarrow{u}^j_n(t,x) = & \left< \nabla\right> f_1 (\left<\nabla\right> ^{-1} \overrightarrow{u}^j_n )+ \left< \nabla\right> f_2 (\left<\nabla\right> ^{-1} \overrightarrow{u}^j_n ),\\
  \quad \overrightarrow{u}^j_n (0)=& \overrightarrow{v}^j_n (0)= T^j_n\overrightarrow{V}^j(\tau^j_n),\quad  u^j_n (0)\in
  \KKK^+,
\end{align*}
where $\tau^j_n = - t^j_n/ (h^j_n)^2$. In order to look for the nonlinear
profile $\overrightarrow{U}^j_{\infty}$ associated to the radial, free concentrating wave
$\left(\overrightarrow{v}^j_n;\; h^j_n, t^j_n \right)$, we
also need undo the group action. We denote
\begin{align*}
   \overrightarrow{u}^j_n(t,x) =& T^j_n
\overrightarrow{U}^j_n\left(\frac{t-t^j_n}{(h^j_n)^2}\right),
\end{align*}
then we have
\begin{align*}
 \left(i\partial_t + \Delta\right) \overrightarrow{U}^j_n = &
\left( \left<\nabla\right>^{j}_{n}\right) f_1 \left( \left(
\left<\nabla\right>^{j}_{n}\right) ^{-1} \overrightarrow{U}^j_n
\right) + h^j_n \cdot \left( \left<\nabla\right>^{j}_{n}\right) f_2
\left( \left( \left<\nabla\right>^{j}_{n}\right) ^{-1}
\overrightarrow{U}^j_n \right),\\
\overrightarrow{U}^j_n(\tau^j_n) =& \overrightarrow{V}^j(\tau^j_n).
\end{align*}

Up to a subsequence, we may assume that there exist
$h^j_{\infty}\in \{0, 1\}$ and $\tau^j_{\infty} \in [-\infty, \infty]$ for
every $j=\{0, \ldots, K\}$, such that
\begin{align*}
h^j_n \rightarrow   \; h^j_{\infty} ,\;\; \text{and}\;\;  \tau^j_n
\rightarrow \; \tau^j_{\infty}  .
\end{align*}
As $n\rightarrow +\infty$, the limit equation of $\overrightarrow{U}^j_n$
is given by
\begin{align*}
 \left(i\partial_t + \Delta\right) \overrightarrow{U}^j_{\infty} = &
\left( \left<\nabla\right>^{j}_{\infty}\right) f_1 \left( \left(
\left<\nabla\right>^{j}_{\infty}\right) ^{-1}
\overrightarrow{U}^j_{\infty} \right) + h^j_{\infty} \cdot \left(
\left<\nabla\right>^{j}_{\infty}\right) f_2 \left( \left(
\left<\nabla\right>^{j}_{\infty}\right) ^{-1}
\overrightarrow{U}^j_{\infty} \right),\\
\overrightarrow{U}^j_\infty(\tau^j_\infty) =&
\overrightarrow{V}^j(\tau^j_\infty) \in L^2(\R^3).
\end{align*}

Let
\begin{align*}
\widehat{U}^j_{\infty}:=
\left(\left<\nabla\right>^j_{\infty}\right)^{-1}\overrightarrow{U}^j_{\infty},
\end{align*}
then
\begin{align}\label{nonlinear profile}
 \left(i\partial_t + \Delta\right) \widehat{U}^j_{\infty} = &
  f_1 \left(\widehat{U}^j_{\infty} \right) + h^j_{\infty} \cdot
  f_2 \left(\widehat{U}^j_{\infty} \right),\\
 \widehat{U}^j_{\infty}(\tau^j_\infty) =& \left(\left<\nabla\right>^j_{\infty}\right)^{-1}
\overrightarrow{V}^j(\tau^j_\infty).
\end{align}

The unique existence of a local radial solution
$\overrightarrow{U}^j_{\infty}$ around $\tau^j_{\infty}$ is known in
all cases, including $h^j_{\infty}=0$ and $\tau^j_{\infty}=\pm
\infty$. $\overrightarrow{U}^j_{\infty} $ on the maximal existence
interval is called the nonlinear profile associated with the radial, free
concentrating wave $\left(\overrightarrow{v}^j_n;\; h^j_n, t^j_n\right)$.

The nonlinear concentrating wave $u^j_{(n)}$ associated with
$\left(\overrightarrow{v}^j_n;\; h^j_n, t^j_n\right)$ is
defined by
\begin{align*}
\overrightarrow{u}^j_{(n)}(t,x)=T^j_n
\overrightarrow{U}^j_{\infty}\left(\frac{t-t^j_n}{(h^j_n)^2}\right),
\end{align*}
then we have
\begin{align*}
 \left(i\partial_t + \Delta\right) \overrightarrow{u}^j_{(n)} = &
\left( \sqrt{|\nabla|^2 +
\left(\frac{h^j_{\infty}}{h^j_n}\right)^2}\right) f_1 \left( \left(
\sqrt{|\nabla|^2 +
\left(\frac{h^j_{\infty}}{h^j_n}\right)^2}\right)^{-1}
\overrightarrow{u}^j_{(n)} \right) \\
& +  \frac{h^j_{\infty}}{h^j_n} \cdot  \left( \sqrt{|\nabla|^2 +
\left(\frac{h^j_{\infty}}{h^j_n}\right)^2}\right) f_2 \left( \left(
\sqrt{|\nabla|^2 +
\left(\frac{h^j_{\infty}}{h^j_n}\right)^2}\right)^{-1}
\overrightarrow{u}^j_{(n)} \right) \\
=& \left<\nabla\right>^j_{\infty} f_1 \left( \left(
\left<\nabla\right>^j_{\infty}\right)^{-1}
\overrightarrow{u}^j_{(n)} \right) +  h^j_{\infty} \cdot
\left<\nabla\right>^j_{\infty} f_2 \left( \left(
\left<\nabla\right>^j_{\infty} \right)^{-1}
\overrightarrow{u}^j_{(n)} \right), \\
 \overrightarrow{u}^j_{(n)}(0) =&
T^j_n\overrightarrow{U}^j_{\infty}(\tau^j_n),
\end{align*}
which implies that
\begin{align*}
\big\|\overrightarrow{u}^j_{(n)}(0) - \overrightarrow{u}^j_{n}(0)
\big\|_{L^2} = & \big\|T^j_n\overrightarrow{U}^j_{\infty}(\tau^j_n)
- T^j_n\overrightarrow{V}^j(\tau^j_n) \big\|_{L^2}
=\big\|
\overrightarrow{U}^j_{\infty}(\tau^j_n) -
 \overrightarrow{V}^j(\tau^j_n) \big\|_{L^2}
\\
\leq & \big\| \overrightarrow{U}^j_{\infty}(\tau^j_n) -
\overrightarrow{U}^j_{\infty}(\tau^j_{\infty}) \big\|_{L^2} + \big\|
 \overrightarrow{V}^j(\tau^j_n)- \overrightarrow{V}^j(\tau^j_{\infty})
 \big\|_{L^2}   \rightarrow  0.
\end{align*}

We denote
\begin{align*}
\overrightarrow{u}^j_{(n)} = \left<\nabla \right>u^j_{(n)}.
\end{align*}
If $h^j_{\infty}=1$, we have $h^j_n\equiv1$, then $ u^j_{(n)} \in
H^1(\R^3)$ is radial and satisfies
\begin{align*}
 \left(i\partial_t + \Delta\right) u^j_{(n)} = f_1(u^j_{(n)}) + f_2
 (u^j_{(n)}).
\end{align*}
If $h^j_{\infty}=0$, then $ u^j_{(n)} \in H^1(\R^3)$ is radial and satisfies
%\begin{align*} \fbox{\rule[-3mm]{0cm}{10mm}$
% \left(i\partial_t + \Delta\right) u^j_{(n)} =
% \frac{|\nabla|}{\left<\nabla\right>}f_1\left(\frac{\left<\nabla\right>}{|\nabla|}u^j_{(n)}\right).$}
%\end{align*}
\begin{align*}
 \left(i\partial_t + \Delta\right) u^j_{(n)} =
 \frac{|\nabla|}{\left<\nabla\right>}f_1\left(\frac{\left<\nabla\right>}{|\nabla|}u^j_{(n)}\right).
\end{align*}

Let $u_n$ be a sequence of (local) radial solutions of \eqref{NLS} with
initial data in $\KKK^+$ at $t=0$, and let $v_n$ be the sequence of
the radial, free solutions with the same initial data. We consider the
linear profile decomposition given by Proposition \ref{L:linear
profile}

\begin{align*}
\overrightarrow{v}_n(t,x) = \sum^{k-1}_{j=0}
\overrightarrow{v}^j_n(t,x) +
\overrightarrow{w}^k_n(t,x),\quad\overrightarrow{v}^j_n (t^j_n) =
T^j_n \varphi^j, \quad  v^j_n(0)\in \KKK^+.
\end{align*}
With each free concentrating wave $\{\overrightarrow{v}^j_n\}_{n\in
\N}$, we associate the nonlinear concentrating wave
$\{\overrightarrow{u}^{j}_{(n)}\}_{n\in \N}$. A nonlinear profile
decomposition of $u_n$ is given by
 \begin{align}\label{profile:nonlinear}
\overrightarrow{u}^{<k}_{(n)}(t,x):=\sum^{k-1}_{j=0}\overrightarrow{u}^j_{(n)}(t,x)
= \sum^{k-1}_{j=0} T^j_n
\overrightarrow{U}^j_{\infty}\left(\frac{t-t^j_n}{(h^j_n)^2}\right).
\end{align}

Since the smallness condition \eqref{small:w weak topology} and the orthogonality condition \eqref{orth:I} ensure that
every nonlinear concentrating wave and the remainder interacts weakly with the others, we will show that $\overrightarrow{u}^{<k}_{(n)}+
\overrightarrow{w}^k_n$ is a good approximation for
$\overrightarrow{u}_n$ provided that each nonlinear profile has the finite
global Strichartz norm.

Now we define the Strichartz norms for the
nonlinear profile decomposition. Let $ST(I)$ and $ST^*(I)$ be the function spaces on $I \times \R^3$
defined as Section \ref{S:long time perturbation}
\begin{align*}
ST(I):=   \left(L^{10}_t \dot B^{1/3}_{90/19, 2} \cap L^{12}_tL^9_x \cap L^6_t\dot B^{1/2}_{18/7,2}\cap
L^{5}_{t,x}\right)&(I\times \R^3),\\
ST^*(I):=   \left( L^2_t \dot B^{1/3}_{18/11,2}  \cap
L^2_t\dot B^{1/2}_{6/5,2}\right)(I\times \R^3)&.
\end{align*}
The Strichartz norm for the nonlinear profile $\widehat{U}^j_{\infty}$ depends on the scaling
$h^j_{\infty}$.
\begin{align*}
ST^j_{\infty}(I):=\begin{cases} ST(I), \quad & \text{for}\; h^j_{\infty}=1, \\
\left(L^{10}_t \dot B^{1/3}_{90/19, 2} \cap L^{12}_tL^9_x \right)(I\times \R^3), \quad & \text{for}\;
h^j_{\infty}=0.
\end{cases}
\end{align*}

\begin{lemma} \label{orth:III}
In the nonlinear profile decomposition
\eqref{profile:nonlinear}. Suppose that for each $j<K$, we have
\begin{align*}
\big\|\widehat{U}^j_{\infty}\big\|_{ST^j_{\infty}(\R)}
+\big\|\overrightarrow{U}^{j}_{\infty}\big\|_{L^{\infty}_tL^2_x(\R^3)}<\infty.
\end{align*}
Then for any finite interval $I$, any $j<K$
and any $k\leq K$, we have
\begin{align}
\varlimsup_{n\rightarrow +\infty}  \big\|u^j_{(n)}\big\|_{ST(I)}
\lesssim & \big\|\widehat{U}^j_{\infty}\big\|_{ST^j_{\infty}(\R)},\label{approximation:ST control:single} \\
\varlimsup_{n\rightarrow +\infty}
\big\|u^{<k}_{(n)}\big\|^2_{ST(I)}\lesssim
&\varlimsup_{n\rightarrow +\infty}  \sum_{j<k}
\big\|u^j_{(n)}\big\|^2_{ST(I)}, \label{approximation:ST control:all}
\end{align}
where the implicit constants do not depend on $I, j$ or $k$. We
also have
\begin{align}
\lim_{n\rightarrow +\infty} \left\| f_1\left(u^{<k}_{(n)}\right)   - \sum_{j<k}
\frac{\left<\nabla\right>^j_{\infty}}{\left<\nabla\right>}  f_1 \left(
\frac{\left<\nabla\right>}{\left<\nabla\right>^j_{\infty}}  u^j_{(n)} \right)  \right\|_{ST^*(I)}=0, \label{orth:ST norm:I}\\
\lim_{n\rightarrow +\infty} \left\| f_2\left(u^{<k}_{(n)}\right)   - \sum_{j<k}  h^j_{\infty}
\frac{\left<\nabla\right>^j_{\infty}}{\left<\nabla\right>}  f_2 \left(
\frac{\left<\nabla\right>}{\left<\nabla\right>^j_{\infty}}  u^j_{(n)} \right)  \right\|_{ST^*(I)}=0.  \label{orth:ST norm:II}
\end{align}
\end{lemma}

\begin{proof} {\bf Proof of \eqref{approximation:ST control:single}.}  By the definitions of $u^j_{(n)}$ and $\widehat{U}^j_{\infty}$, we know that
\begin{align*}
u^j_{(n)}(t,x)= & \left<\nabla\right>^{-1}\overrightarrow{u}^j_{(n)}(t,x)
=\left<\nabla\right>^{-1}T^j_n\overrightarrow{U}^j_{\infty}\left(\frac{t-t^j_n}{(h^j_n)^2}\right)\\
=&\left<\nabla\right>^{-1}T^j_n \left<\nabla\right>^{j}_{\infty}
\widehat{U}^j_{\infty}\left(\frac{t-t^j_n}{(h^j_n)^2}\right) =
h^j_nT^j_n\frac{\left<\nabla\right>^{j}_\infty}{ \left<\nabla\right>^j_n }
\widehat{U}^j_{\infty}\left(\frac{t-t^j_n}{(h^j_n)^2}\right).
\end{align*}
For the
case $h^j_{\infty}=1$, we have $u^j_{(n)}(t,x)=\widehat{U}^j_{\infty}(t-t^j_n, x)$, hence \eqref{approximation:ST control:single}
 is trivial. For the case $h^j_{\infty}=0$, by the above relation between $u^j_{(n)}$ and
$\widehat{U}^j_{\infty}$, we have
\begin{align*}
\big\|u^j_{(n)}\big\|_{ \left(L^{10}_t \dot B^{1/3}_{90/19, 2} \cap L^{12}_tL^9_x \right)(I\times
\R^3) } \lesssim & \left\|\frac{|\nabla|}{\left<\nabla\right>^j_n}
\widehat{U}^j_{\infty} \right\|_{ \left(L^{10}_t \dot B^{1/3}_{90/19, 2} \cap L^{12}_tL^9_x \right)(\R
\times \R^3) }\\
 \lesssim & \big\|  \widehat{U}^j_{\infty} \big\|_{
\left(L^{10}_t \dot B^{1/3}_{90/19, 2} \cap L^{12}_tL^9_x \right)( \R \times \R^3) },
\end{align*}
and
\begin{align*}
\big\|u^j_{(n)}\big\|_{  L^6_t\dot B^{1/2}_{18/7,2} (I\times \R^3)
}\lesssim & |I|^{\frac{1}{12}}
\big\|u^j_{(n)}\big\|_{L^{12}_t \dot B^{1/2}_{18/7,2} }
\lesssim   |I|^{\frac{1}{12}} (h^j_n)^{\frac{1}{3}}
\big\|\widehat{U}^j_{\infty}\big\|_{L^{12}_t \dot
B^{\frac{5}{6}}_{18/7, 2}} \rightarrow 0, \\
\big\|u^j_{(n)}\big\|_{
L^{5}_{t,x} (I\times \R^3)
}\lesssim & |I|^{\frac{7}{60}}
\big\|u^j_{(n)}\big\|_{L^{12}_t L^{5}_x}
\lesssim   |I|^{\frac{7}{60}} (h^j_n)^{\frac{4}{15}}
\big\|\widehat{U}^j_{\infty}\big\|_{L^{12}_t \dot
B^{\frac{4}{15}}_{5, 2}} \rightarrow 0.
\end{align*}
where we use the fact that the boundedness of $\widehat{U}^j_{\infty} $ in $ L^{10}_t \dot B^{1/3}_{90/19, 2} \cap    L^{12}_tL^9_x  \cap L^{\infty}\dot H^1$
implies its boundedness in $L^{12}_t \dot
B^{\frac{5}{6}}_{18/7, 2}\cap L^{12}_t \dot
B^{\frac{4}{15}}_{5, 2}$ by \eqref{nonlinear profile}.

{\noindent \bf Proof of \eqref{approximation:ST control:all}.}  We estimate the left hand side of  \eqref{approximation:ST control:all} by
\begin{align*}
 \big\|u^{<k}_{(n)}\big\|^2_{ST(I)} = & \left\|\sum_{j<k,
 h^j_{\infty}=1} u^j_{(n)} + \sum_{j<k,  h^j_{\infty}=0} u^j_{(n)}
\right\|^2_{ST(I)} \\
\lesssim & \left\|\sum_{j<k,
 h^j_{\infty}=1} u^j_{(n)}\right\|^2_{ST(I)}  +   \left\| \sum_{j<k,  h^j_{\infty}=0} u^j_{(n)}
\right\|^2_{ST(I)}.
\end{align*}
For the case $h^j_{\infty}=1$. Define $\widehat{U}^j_{\infty, R}$ and
$u^j_{(n), R}$ by
\begin{align*}
\widehat{U}^j_{\infty, R}(t,x)=\chi_{R}(t,x)\widehat{U}^j_{\infty}(t,x),\quad
u^j_{(n), R}(t,x)=T^j_n \widehat{U}^j_{\infty, R}(t-t^j_n),
\end{align*}
where $\chi_R$ is the cut-off function as in Remark \ref{rem:noempty}. Then we have
\begin{align*}
\left\|\sum_{j<k,
 h^j_{\infty}=1} u^j_{(n)}\right\|^2_{ST(I)} \lesssim &  \left\| \sum_{j<k, h^j_{\infty}=1}u^j_{(n), R}\right\|^2_{ST(I)}+  \left\|\sum_{j<k,
 h^j_{\infty}=1} u^j_{(n)} - \sum_{j<k, h^j_{\infty}=1}u^j_{(n), R}\right\|^2_{ST(I)}.
\end{align*}
On one hand, we know that
\begin{align*}
 \left\|\sum_{j<k,
 h^j_{\infty}=1} u^j_{(n)} - \sum_{j<k, h^j_{\infty}=1}u^j_{(n),
 R}\right\|_{ST(I)}  \leq  \sum_{j<k, h^j_{\infty}=1} \big\|(1-\chi_{R}) u^j_{(n)}
 \big\|_{ST(I)} \rightarrow 0,
\end{align*}
as $R\rightarrow +\infty$. On the other hand, by \eqref{orth:I} and the similar orthogonality analysis as in \cite{IbrMN:f:NLKG}, we
know that
\begin{align*}
\varlimsup_{n\rightarrow +\infty} \left\| \sum_{j<k, h^j_{\infty}=1}u^j_{(n), R}\right\|^2_{ST(I)}
 \lesssim \varlimsup_{n\rightarrow +\infty}  \sum_{j<k, h^j_{\infty}=1}  \left\| u^j_{(n),
 R}\right\|^2_{ST(I)}\lesssim \varlimsup_{n\rightarrow +\infty}  \sum_{j<k, h^j_{\infty}=1}  \left\| u^j_{(n) }\right\|^2_{ST(I)}.
\end{align*}

For the case $h^j_{\infty}=0$, On one hand, by $h^j_n\rightarrow 0$, we have
\begin{align*}
\varlimsup_{n\rightarrow+\infty}  \left\| \sum_{j<k, h^j_{\infty}=0}
u^j_{(n)} \right\|_{\left( L^6_t\dot B^{1/2}_{18/7,2}\cap
L^{5}_{t,x}\right)(I\times \R^3)} =0.
\end{align*}
On the other hand, by \eqref{orth:I} and the analogue approximation analysis as in \cite{IbrMN:f:NLKG}, we have
\begin{align*}
\varlimsup_{n\rightarrow +\infty}  \left\| \sum_{j<k, h^j_{\infty}=0}
u^j_{(n)} \right\|^2_{ L^{10}_t \dot B^{1/3}_{90/19, 2}  (I\times
\R^3) } \lesssim & \varlimsup_{n\rightarrow +\infty}  \sum_{j<k, h^j_{\infty}=0} \left\|  u^j_{(n)}
\right\|^2_{  L^{10}_t \dot B^{1/3}_{90/19, 2}  (I\times \R^3) },\\
\varlimsup_{n\rightarrow +\infty}  \left\| \sum_{j<k, h^j_{\infty}=0}
u^j_{(n)} \right\|^2_{ L^{12}_tL^9_x  (I\times
\R^3) } \lesssim & \varlimsup_{n\rightarrow +\infty}  \sum_{j<k, h^j_{\infty}=0} \left\|  u^j_{(n)}
\right\|^2_{ L^{12}_tL^9_x  (I\times \R^3) }.
\end{align*}

{\noindent \bf Proof of \eqref{orth:ST norm:I}.} Let $\displaystyle
u^{<k}_{<n>}(t,x):=
 \sum_{j<k} u^j_{<n>}(t,x)$ where
\begin{align*}
u^j_{<n>}(t,x):=   \frac{\left<\nabla\right>}{\left<\nabla\right>^j_{\infty} }
u^j_{(n)} = \frac{1}{\left<\nabla\right>^j_{\infty}}
\overrightarrow{u}^j_{(n)}
 =   \frac{1}{
\left<\nabla\right>^j_{\infty}} T^j_n
\overrightarrow{U}^j_{\infty}\left(\frac{t-t^j_n}{(h^j_n)^2}\right)
= h^j_n T^j_n
\widehat{U}^j_{\infty}\left(\frac{t-t^j_n}{(h^j_n)^2}\right),
\end{align*}
and
\begin{align*}
u^j_{(n)}(t,x)=h^j_n T^j_n \frac{\left<\nabla\right>^j_{\infty}}{\left<\nabla\right>^j_n}
\widehat{U}^j_{\infty}\left(\frac{t-t^j_n}{(h^j_n)^2}\right).
\end{align*}
Then we have
\begin{align*}
& \left\| f_1\left(u^{<k}_{(n)}\right)   - \sum_{j<k}
\frac{\left<\nabla\right>^j_{\infty}}{\left<\nabla\right>}  f_1 \left(
\frac{\left<\nabla\right>}{\left<\nabla\right>^j_{\infty}}  u^j_{(n)} \right)  \right\|_{ST^*} \\
&\leq  \left\| f_1\left(u^{<k}_{(n)}\right) - f_1\left(u^{<k}_{\left<n\right>} \right) \right\|_{ST^*} +
\left\| f_1\left(u^{<k}_{\left<n\right>}\right)- \sum_{j<k} f_1 \left(
u^j_{\left<n\right>} \right)  \right\|_{ST^*} \\
&\quad + \left\| \sum_{j<k}
f_1 \left(
u^j_{\left<n\right>} \right)-\sum_{j<k}
\frac{\left<\nabla\right>^j_{\infty}}{\left<\nabla\right>}  f_1 \left(
u^j_{\left<n\right>} \right)  \right\|_{ST^*}\\
 & \leq\left\| f_1\left(u^{<k}_{(n)}\right) - f_1\left(u^{<k}_{\left<n\right>} \right) \right\|_{ST^*} +
\left\| f_1\left(u^{<k}_{\left<n\right>}\right)- \sum_{j<k} f_1 \left(
u^j_{\left<n\right>} \right)  \right\|_{ST^*} \\
& \quad + \left\| \sum_{j<k, h^j_{\infty}=0}
f_1 \left(
u^j_{\left<n\right>} \right)-\sum_{j<k, h^j_{\infty}=0}
\frac{|\nabla|}{\left<\nabla\right>}  f_1 \left(
u^j_{\left<n\right>} \right)  \right\|_{ST^*}.
\end{align*}
By \eqref{orth:I} and the approximation argument in \cite{IbrMN:f:NLKG}, we have
\begin{align*}
 \left\| f_1\left(u^{<k}_{(n)}\right) - f_1\left(u^{<k}_{\left<n\right>} \right) \right\|_{ST^*} +
\left\| f_1\left(u^{<k}_{\left<n\right>}\right)- \sum_{j<k} f_1 \left(
u^j_{\left<n\right>} \right)  \right\|_{ST^*} \longrightarrow 0
\end{align*}
as $n\rightarrow +\infty$. In addition, by $h^j_n\rightarrow 0$ as $n\rightarrow +\infty$, we have
\begin{align*}
 \left\| \sum_{j<k, h^j_{\infty}=0} \left(1-\frac{|\nabla|}{\left<\nabla\right>} \right)
f_1 \left(
u^j_{\left<n\right>} \right)   \right\|_{L^2_t \dot B^{1/3}_{18/11,2}  }
  =&  \left\| \left(1- \frac{|\nabla|}{\left<\nabla\right>^j_n}\right) \sum_{j<k, h^j_{\infty}=0}
f_1 \left(\widehat{U}^j_{\infty}
 \right)   \right\|_{L^2_t \dot B^{1/3}_{18/11,2}  }
\longrightarrow    0,\\
 \left\| \sum_{j<k, h^j_{\infty}=0} \left(1-\frac{|\nabla|}{\left<\nabla\right>} \right)
f_1 \left(
u^j_{\left<n\right>} \right)  \right\|_{
L^2_t\dot B^{1/2}_{6/5,2}}
 = & \left(h^j_n\right)^{1/2}\left\| \left(1- \frac{|\nabla|}{\left<\nabla\right>^j_n}\right) \sum_{j<k, h^j_{\infty}=0}
f_1 \left(\widehat{U}^j_{\infty}
 \right)   \right\|_{
L^2_t\dot B^{1/2}_{6/5,2}}
\longrightarrow   0,
\end{align*}
as $n\rightarrow +\infty$. Therefore, we have
\begin{align*}
\lim_{n\rightarrow +\infty} \left\| f_1\left(u^{<k}_{(n)}\right)   - \sum_{j<k}
\frac{\left<\nabla\right>^j_{\infty}}{\left<\nabla\right>}  f_1 \left(
\frac{\left<\nabla\right>}{\left<\nabla\right>^j_{\infty}}  u^j_{(n)} \right)  \right\|_{ST^*}=0.
\end{align*}

{\noindent \bf Proof of \eqref{orth:ST norm:II}.} Note that
\begin{align*}
& \left\| f_2\left(u^{<k}_{(n)}\right)   - \sum_{j<k}  h^j_{\infty}
\frac{\left<\nabla\right>^j_{\infty}}{\left<\nabla\right>}  f_2 \left(
\frac{\left<\nabla\right>}{\left<\nabla\right>^j_{\infty}}  u^j_{(n)} \right)  \right\|_{ST^*} \\
 & \leq \left\| f_2\left(u^{<k}_{(n)}\right) -  f_2\left(u^{<k}_{\left<n\right>}\right)  \right\|_{ST^*}
+ \left\|  f_2\left(u^{<k}_{\left<n\right>}\right)- \sum_{j<k } f_2\left(u^{j}_{\left<n\right>}\right)  \right\|_{ST^*}
  + \left\|\sum_{j<k, h^j_{\infty}=0} f_2\left(u^{j}_{\left<n\right>}\right)   \right\|_{ST^*}.
\end{align*}

By the analogue analysis, we have
\begin{align*}
 \left\| f_2\left(u^{<k}_{(n)}\right) -  f_2\left(u^{<k}_{\left<n\right>}\right)  \right\|_{ST^*}
+ \left\|  f_2\left(u^{<k}_{\left<n\right>}\right)- \sum_{j<k} f_2\left(u^{j}_{\left<n\right>}\right)  \right\|_{ST^*} \longrightarrow 0,
\end{align*}
and
\begin{align*}
  \left\|\sum_{j<k, h^j_{\infty}=0} f_2\left(u^{j}_{\left<n\right>}\right)   \right\|_{ST^*} \longrightarrow 0
\end{align*}
as $n \rightarrow +\infty$. Hence, we obtain
\begin{align*}
\lim_{n\rightarrow +\infty} \left\| f_2\left(u^{<k}_{(n)}\right)   - \sum_{j<k}  h^j_{\infty}
\frac{\left<\nabla\right>^j_{\infty}}{\left<\nabla\right>}  f_2 \left(
\frac{\left<\nabla\right>}{\left<\nabla\right>^j_{\infty}}  u^j_{(n)} \right)  \right\|_{ST^*}=0.
\end{align*}
These complete the proof.
\end{proof}

After this preliminaries, we now show that $\overrightarrow{u}^{<k}_{(n)}+
\overrightarrow{w}^k_n$ is a good approximation for
$\overrightarrow{u}_n$ provided that each nonlinear profile has finite
global Strichartz norm.

\begin{proposition}\label{approximation}
Let $u_n$ be a sequence of local, radial solutions of \eqref{NLS} around
$t=0$ in $\KKK^+$ satisfying \begin{align*}
M\left(u_n\right)< \infty, \quad \varlimsup_{n\rightarrow
\infty} E(u_n)<m.
\end{align*} Suppose that in the nonlinear profile decomposition
\eqref{profile:nonlinear}, every nonlinear profile
$\widehat{U}^j_{\infty}$ has finite global Strichartz and energy
norms we have
\begin{align*}
\big\|\widehat{U}^j_{\infty}\big\|_{ST^j_{\infty}(\R)}
+\big\|\overrightarrow{U}^{j}_{\infty}\big\|_{L^{\infty}_tL^2_x(\R^3)}<\infty.
\end{align*}
 Then $u_n$ is bounded for large $n$ in the Strichartz and the
 energy norms
\begin{align*}\varlimsup_{n\rightarrow \infty}
\big\|u_n\big\|_{ST(\R)} +
\big\|\overrightarrow{u}_n\big\|_{L^{\infty}_tL^2_x(\R)} < \infty.
\end{align*}
\end{proposition}

\begin{proof}We only need to verify the condition of Proposition \ref{stability}. Note that $u^{<k}_{(n)} + w^k_n$ satisfies that
\begin{align*}
  & \left(i \partial_t + \Delta\right) \left(u^{<k}_{(n)} + w^k_n\right)=f_1\left(u^{<k}_{(n)} + w^k_n\right) + f_2\left( u^{<k}_{(n)} + w^k_n\right)\\
 &\qquad + f_1\left(u^{<k}_{(n)}\right) - f_1\left(u^{<k}_{(n)}+ w^k_n \right)
  + f_2\left(u^{<k}_{(n)}\right) - f_2\left(u^{<k}_{(n)}+ w^k_n \right)\\
& \qquad+ \sum_{j<k}
\frac{\left<\nabla\right>^j_{\infty}}{\left<\nabla\right>}  f_1 \left(
\frac{\left<\nabla\right>}{\left<\nabla\right>^j_{\infty}}  u^j_{(n)} \right) - f_1\left(u^{<k}_{(n)}\right)   + \sum_{j<k}  h^j_{\infty}
\frac{\left<\nabla\right>^j_{\infty}}{\left<\nabla\right>}  f_2 \left(
\frac{\left<\nabla\right>}{\left<\nabla\right>^j_{\infty}}  u^j_{(n)} \right) -  f_2\left(u^{<k}_{(n)}\right).
\end{align*}

First, by the construction of $\overrightarrow{u}^{<k}_{(n)}$, we know that
\begin{align*}
\left\| \left(\overrightarrow{u}^{<k}_{(n)} (0) + \overrightarrow{w}^k_n(0)\right) -\overrightarrow{u}_n(0) \right\|_{ L^2_x}
\leq \sum_{j<k} \left\|\overrightarrow{u}^j_{(n)}(0) -\overrightarrow{u}^j_n(0) \right\|_{ L^2_x} \rightarrow 0,
\end{align*}
as $n\rightarrow +\infty$, which also implies that for large $n$, we have
\begin{align*}
\left\| \overrightarrow{u}^{<k}_{(n)} (0) + \overrightarrow{w}^k_n(0)  \right\|_{ L^2_x} \leq E_0.
\end{align*}

Next, by the linear profile decomposition in Proposition \ref{L:linear profile}, we know that
\begin{align*}
\big\|u_n(0)\big\|^2_{L^2}= &  \left\|v_n(0)\right\|^2_{L^2_x}
= \sum_{j<k}\left\|v^j_n(0)\right\|^2_{L^2_x} + \left\|w^k_n(0)\right\|^2_{L^2_x} + o_n(1)\\
\geq &  \sum_{j<k}\left\|v^j_n(0)\right\|^2_{L^2_x} + o_n(1)
=     \sum_{j<k}\left\|u^j_{(n)}(0)\right\|^2_{L^2_x} + o_n(1),
\\
\left\|u_n(0)\right\|^2_{\dot H^1_x}=& \left\|v_n(0)\right\|^2_{\dot H^1_x}
= \sum_{j<k}\left\|v^j_n(0)\right\|^2_{\dot H^1_x} + \left\|w^k_n(0)\right\|^2_{\dot H^1_x} + o_n(1)\\
\geq &  \sum_{j<k}\left\|v^j_n(0)\right\|^2_{\dot H^1_x} + o_n(1)
=     \sum_{j<k}\left\|u^j_{(n)}(0)\right\|^2_{\dot H^1_x} + o_n(1),
\end{align*}
which means except for a finite set $J\subset \N$, the energy of $u^j_{(n)}$ with $j\not \in J$ is smaller than the iteration
threshold, hence we have
\begin{align*}
\big\|u^j_{(n)}\big\|_{ST(\R)} \lesssim \big\|\overrightarrow{u}^j_{(n)}(0)\big\|_{L^2_x},
\end{align*}
thus, for any finite interval $I$, by Lemma \ref{orth:III}, we have
\begin{align*}
\sup_k \varlimsup_{n\rightarrow +\infty} \big\|u^{<k}_{(n)}\big\|^2_{ST(I)} \lesssim &  \sup_k \varlimsup_{n\rightarrow +\infty}\sum_{j<k} \big\|u^{j}_{(n)}\big\|^2_{ST(I)}\\
= & \sup_k \varlimsup_{n\rightarrow +\infty} \left[\sum_{j<k, j\in J} \big\|u^{j}_{(n)}\big\|^2_{ST(I)} + \sum_{j<k, j\not \in J} \big\|u^{j}_{(n)}\big\|^2_{ST(I)}\right]\\
\lesssim & \sum_{j<k, j\in J} \big\|\widehat{U}^{j}_{\infty}\big\|^2_{ST^j_{\infty}(I)} + \sup_k \varlimsup_{n\rightarrow +\infty} \sum_{j<k, j\not \in J} \big\|\overrightarrow{u}^j_{(n)}(0)\big\|^2_{L^2_x}\\
< & \infty.
\end{align*}
This together with the Strichartz estimate for $w^k_n$ implies that
\begin{align*}
\sup_k \varlimsup_{n\rightarrow +\infty} \big\|u^{<k}_{(n)} + w^k_n\big\|^2_{ST(I)} <  \infty.
\end{align*}

Last we need show the nonlinear perturbation is small in some sense. By Proposition \ref{L:linear profile} and Lemma \ref{orth:III}, we have
\begin{align*}
& \left\| f_1\left(u^{<k}_{(n)}\right) - f_1\left(u^{<k}_{(n)}+ w^k_n \right)  \right\|_{ST^*(I)} \rightarrow 0 ,\\
 &  \left\| f_2\left(u^{<k}_{(n)}\right) - f_2\left(u^{<k}_{(n)}+ w^k_n \right) \right\|_{ST^*(I)} \rightarrow 0,
 \end{align*}
and
\begin{align*}
&  \left\| \sum_{j<k}
\frac{\left<\nabla\right>^j_{\infty}}{\left<\nabla\right>}  f_1 \left(
\frac{\left<\nabla\right>}{\left<\nabla\right>^j_{\infty}}  u^j_{(n)} \right) - f_1\left(u^{<k}_{(n)}\right)  \right\|_{ST^*(I)} \rightarrow 0,  \\
&   \left\| \sum_{j<k}  h^j_{\infty}
\frac{\left<\nabla\right>^j_{\infty}}{\left<\nabla\right>}  f_2 \left(
\frac{\left<\nabla\right>}{\left<\nabla\right>^j_{\infty}}  u^j_{(n)} \right) -  f_2\left(u^{<k}_{(n)}\right)  \right\|_{ST^*(I)} \rightarrow 0,
\end{align*}
as $n\rightarrow +\infty$. Therefore, by Proposition \ref{stability}, we can obtain the desired result, which concludes the proof.
\end{proof}

%%%%%%%%%%%%%%%%%%%%%%%%%%%%%%%%%%%%%%%%%%%%%%%%%%%%%%%%%%%%%%%%%%%%%%%%%%%%%%%%%%%%%%%%%%%
%
%
%                                   Section six
%
%
%%%%%%%%%%%%%%%%%%%%%%%%%%%%%%%%%%%%%%%%%%%%%%%%%%%%%%%%%%%%%%%%%%%%%%%%%%%%%%%%%%%%%%%%%%%

\section{Part II: GWP and Scattering for $\KKK^+$}\label{S:GWP-Scattering}
After the stability analysis of the
scattering solution of \eqref{NLS} and the compactness analysis (linear and nonlinear profile decompositions) of
a sequence of the radial solutions of \eqref{NLS} in the energy space.
We now use them to show the scattering result of Theorem
\ref{theorem} by contradiction.

Let $E^*$ be the threshold for the uniform Strichartz norm bound,
i.e.,
\begin{align*}
E^*:=\sup\{A>0, ST(A)<\infty\}
\end{align*}
where $ST(A)$ denotes the supremum of $\big\|u\big\|_{ST(I)}$ for any
strong radial solution $u$ of \eqref{NLS} in $\KKK^+$ on any interval $I$ satisfying
$E(u)\leq A$, $M(u)<\infty$.

The small solution scattering theory gives us $E^*>0$.

Now we are going to show that $E^*\geq m$ by contradiction. From now on, suppose that $E^*\geq m$ fails,
that is, we assume that
\begin{align}\label{contradiction:J}
E^*<m.
\end{align}

\subsection{Existence of a critical element} In this subsection,
 by the profile decomposition and the stability theory of the scattering solution of \eqref{NLS}, we
show the existence of the critical element, which is the radial, energy solution of \eqref{NLS}
with the smallness energy $E^*$ and infinite Strichartz norm.

By the definition of $E^*$ and the fact that $E^*<m$, there exist
a sequence of radial solutions $\{u_n\}_{n\in \N}$  of \eqref{NLS} in $\KKK^+$, which have the maximal existence interval $I_n$ and satisfy that
\begin{align*}
M(u_n)< \infty,
\quad E(u_n)\rightarrow E^*<m, \quad \big\|u_n\big\|_{ST(I_n)}\rightarrow
+\infty, \quad \text{as}\;\; n\rightarrow +\infty,
\end{align*}
then we have $\big\|u_n\big\|_{H^1}<\infty$ by Lemma \ref{free-energ-equiva}. By the compact argument (profile decomposition) and the stability theory, we can show that

\begin{theorem}\label{APS:existence}
Let $u_n$ be a sequence of radial solutions of \eqref{NLS} in $\KKK^+$ on
$I_n\subset \R$ satisfying
\begin{align*}
M(u_n)<\infty,
\quad E(u_n)\rightarrow E^*<m, \quad \big\|u_n\big\|_{ST(I_n)}\rightarrow
+\infty, \quad \text{as}\;\; n\rightarrow +\infty.
\end{align*}
Then there exists a global, radial solution $u_c$ of \eqref{NLS} in $\KKK^+$
satisfying
\begin{align*}
E(u_c)=E^*<m, \quad K(u_c) > 0, \quad \big\|u_c\big\|_{ST(\R)}=\infty.
\end{align*}
In addition, there are a sequence $t_n\in \R$ and radial function
$\varphi \in L^2(\R^3)$ such that, up to a subsequence,  we have as $n\rightarrow +\infty$,
\begin{align}\label{compact:initial}
\left\|\frac{|\nabla|}{\left<\nabla \right>}\Big(\overrightarrow{u}_n(0, x) -
e^{-it_n\Delta}\varphi(x )\Big)\right\|_{L^2}\rightarrow 0.
\end{align}
\end{theorem}
\begin{proof} By the time translation symmetry of \eqref{NLS}, we can translate $u_n$ in $t$ such that $0\in I_n$ for all $n$.
Then by the linear and nonlinear profile decomposition of $u_n$, we have
\begin{align*}
e^{it\Delta }\overrightarrow{u}_n(0,x) = &
\sum_{j<k}\overrightarrow{v}^j_n(t,x) + \overrightarrow{w}^k_n(t,x),
 \quad \overrightarrow{v}^j_n(t,x) =
e^{i(t-t^j_n)\Delta}T^j_n\varphi^j,\\
\overrightarrow{u}^{<k}_{(n)}(t,x) =  & \sum_{j<k}
\overrightarrow{u}^j_{(n)}(t,x),  \quad
\overrightarrow{u}^j_{(n)}(t,x)=T^j_n\overrightarrow{U}^j_{\infty}\left(\frac{t-t^j_n }{(h^j_n)^2}\right),\\
&  \big\|\overrightarrow{u}^j_{(n)}(0) -\overrightarrow{v}^j_n(0)\big\|_{L^2}\rightarrow
0.
\end{align*}

By Proposition \ref{decomp:stable:K} and the following observations that
\begin{enumerate}
\item Every radial solution of \eqref{NLS} in $\KKK^+$ with the
energy less than $E^*$ has global finite Strichartz norm by the
definition of $E^*$.

\item Lemma \ref{approximation} precludes that all the nonlinear profiles
$\overrightarrow{U}^j_{\infty}$ have finite global Strichartz norm.
\end{enumerate}
we deduce that there is only one radial profile and
\begin{align*}
E(u^0_{(n)}(0)) \rightarrow E^*, \quad u^0_{(n)}(0)\in \KKK^+, \quad
\big\|\widehat{U}^{0}_{\infty}\big\|_{ST^{0}_{\infty}(I)}=\infty, \quad
\big\|w^1_{n}\big\|_{L^{\infty}_t\dot H^1_x}\rightarrow 0.
\end{align*}

If $h^0_{n}\rightarrow 0$, then
$\widehat{U}^0_{\infty}=|\nabla|^{-1} \overrightarrow{U}^0_{\infty}$
solves the $\dot H^1$-critical NLS
\begin{align*}
\left(i\partial_t + \Delta \right)\widehat{U}^0_{\infty} = f_1
(\widehat{U}^0_{\infty})
\end{align*}
and satisfies
\begin{align*}
E^c\left(\widehat{U}^0_{\infty}(\tau^0_{\infty})\right)=E^* <m,
\;\; K^c\left(\widehat{U}^0_{\infty}(\tau^0_{\infty})\right)\geq 0,
\;\; \big\|\widehat{U}^0_{\infty}\big\|_{\left(L^{10}_t \dot B^{1/3}_{90/19, 2} \cap L^{12}_tL^9_x \right)(I \times \R^3) }=\infty.
\end{align*}
However, it is in contradiction with Kenig-Merle's result\footnote{By the global $L^{10}_{t,x}$
estimate of solution $u$ of \eqref{NLS:focusing critical}, we can obtain the global $L^q_t\dot W^{1, r}_x$
estimate of $u$ for any Schr\"{o}dinger $L^2$-admissible pair $(q,r)$.} in \cite{KenM:NLS:GWP}. Hence $h^0_n\equiv1$,
which implies \eqref{compact:initial}.

Now we show that $\widehat{U}^0_{\infty}=\left<\nabla\right>^{-1}
\overrightarrow{U}^j_{\infty}$ is a global solution, which is the consequence of the compactness of \eqref{compact:initial}.
Suppose not, then we can choose a sequence $t_n \in \R$ which
approaches the maximal existence time. Since
$\widehat{U}^0_{\infty}(t+t_n)$ satisfies the assumption of this
theorem, then applying the above argument to it, we obtain that for some $\psi\in L^2$ and another sequence $t'_n \in
\R $, as $n\rightarrow +\infty$
\begin{align}\label{compact:t_n}
\left\|\frac{|\nabla|}{\left<\nabla\right>}\left(\overrightarrow{U}^0_{\infty}(t_n)-e^{-it'_n\Delta}\psi(x )\right)\right\|_{L^2}\rightarrow
0.
\end{align}
 Let
$
\overrightarrow{v}(t):=e^{it\Delta}\psi.
$
For any $\varepsilon>0$, there exist $\delta>0$ with $I=[-\delta,
\delta]$ such that
\begin{align*}
\big\|\left<\nabla\right>^{-1}\overrightarrow{v}(t-t'_n)\big\|_{ST(I)}\leq
\varepsilon, \end{align*} which together with \eqref{compact:t_n}
implies that for sufficiently large $n$
\begin{align*}
\big\|\left<\nabla\right>^{-1}
e^{it\Delta}\overrightarrow{U}^0_{\infty}(t_n)\big\|_{ST(I)} \leq
 \varepsilon.
\end{align*}
If $\varepsilon$ is small enough, this implies that the solution
$\widehat{U}^0_{\infty}$ exists on $[t_n-\delta, t_n+\delta]$ for
large $n$ by the small data theory. This contradicts the choice of
$t_n$. Hence $\widehat{U}^0_{\infty}$ is a global solution and it is just the desired
critical element $u_c$. By Proposition \ref{threshold-energy}, we know that $K(u_c)>0$.
\end{proof}

\subsection{Compactness of the critical element}
In order to preclude the critical element, we need obtain some
useful properties about the critical element. In the following
subsections, we establish some properties about the critical element by its minimal energy with infinite Strichartz norm,
especially its compactness and its consequence. Since \eqref{NLS} is
symmetric in $t$, we may assume that
\begin{align}\label{critical:infinity Strich}
\big\| u_c\big\|_{ST(0, +\infty)}=\infty,
\end{align}
we call it a forward critical element.

\begin{proposition}\label{compact:APS}
Let $u_c$ be a forward critical element. Then the set
\begin{align*}
\{u_c(t, x); 0<t<\infty\}
\end{align*}
is precompact in $\dot H^s$ for any $s\in(0, 1]$.
\end{proposition}
\begin{proof} By the conservation of the mass, it suffices to prove the precompactness of
$u_c(t_n)\}$ in $\dot H^1$ for any positive time $t_1,
t_2, \ldots$. If $t_n$ converges, then it is trivial from the continuity in $t$.

If $t_n \rightarrow +\infty$. Applying Theorem \ref{APS:existence}
to the sequence of solutions $\overrightarrow{u}_c(t+t_n)$, we get another sequence
$t'_n \in \R$ and radial function $\varphi\in L^2$ such that
\begin{align*}
\frac{|\nabla|}{\left<\nabla\right>}\left(\overrightarrow{u}_c(t_n,x) - e^{-it'_n\Delta}\varphi(x )\right)
\rightarrow 0 \quad \text{in}\;\; L^2.
\end{align*}
\begin{enumerate}
\item If $t'_n \rightarrow -\infty$, then we have
\begin{align*}
\big\|\left<\nabla\right>^{-1}e^{it\Delta}\overrightarrow{u}_c(t_n)\big\|_{ST(0,
+\infty)} =
\big\|\left<\nabla\right>^{-1}e^{it\Delta}\varphi\big\|_{ST(-t'_n,
+\infty)} + o_n(1) \rightarrow 0.
\end{align*}
Hence $u_c$ can solve \eqref{NLS} for $t>t_n$ with large $n$
globally by iteration with small Strichartz norms, which contradicts
\eqref{critical:infinity Strich}.

\item If $t'_n \rightarrow +\infty$, then we have
\begin{align*}
\big\|\left<\nabla\right>^{-1}e^{it\Delta}\overrightarrow{u}_c(t_n)\big\|_{ST(-\infty,
0)} =
\big\|\left<\nabla\right>^{-1}e^{it\Delta}\varphi\big\|_{ST(-\infty,
-t'_n)} + o_n(1) \rightarrow 0
\end{align*}
Hence $u_c$ can solve \eqref{NLS} for $t<t_n$ with large $n$
with vanishing Strichartz norms, which implies $u_c=0$ by taking the
limit, which is a contradiction.
\end{enumerate}

Thus $t'_n$ is bounded, which implies that $t'_n$ is precompact, so is $u_c(t_n,
x )$ in $\dot H^1$.
\end{proof}

As a consequence, the energy of $u_c$ stays within a fixed
radius for all positive time, modulo arbitrarily small rest. More
precisely, we define the exterior energy by
\begin{align*}
E_{R }(u;t)=\int_{|x |\geq R} \Big( \big|\nabla u(t,x)\big| ^2   + \big|u(t,x)\big|^4 +
\big|u(t,x)\big|^6 \Big) \; dx
\end{align*}
for any $R>0$. Then we have
\begin{corollary} \label{compact:almost periodicity}
Let $u_c$ be a forward critical element. then for any $\varepsilon$,
there exist $R_0(\varepsilon)>0$ such that
\begin{align*}
E_{R_0 }(u_c; t) \leq \varepsilon E(u_c), \; \text{for any}\; t > 0.
\end{align*}
\end{corollary}

\subsection{Death of the critical element}
We are in a position to preclude the soliton-like solution by
a truncated Virial identity.% First we show the lower boundedness of the kinetic energy of the critical element.
%
%\begin{lemma}\label{compact:phenomena} Let $u_c$ be a forward critical element. Then there
%exists $C>0$ such that
%\begin{align*}
%\big\|u_c(t)\big\|^2_{L^2} \leq C \big\|\nabla u_c(t)\big\|^2_{L^2}
%\end{align*}
%for all $t>0$.
%\end{lemma}
%\begin{proof} We prove it by contradiction. Suppose there exists a
%sequence $t_n\geq 0$ such that
%\begin{align*}
%\big\|u(t_n)\big\|^2_{L^2} \geq n \big\|\nabla u(t_n)\big\|^2_{L^2}.
%\end{align*}
%By the mass conservation, it follows that
%$
%\big\|\nabla u(t_n)\big\|^2_{L^2} \rightarrow 0,
%$ as $n\rightarrow +\infty$. By Theorem \ref{compact:APS},  we know that $u(t_n)\rightarrow 0$ in
%$H^1$, after passing to a subsequence. Hence $J^q(u(t_n))
%\rightarrow 0$, which contradicts with the energy equivalence, Lemma
%\ref{free-energ-equiva}.
%\end{proof}

\begin{theorem}
The critical element $u_c$ of \eqref{NLS} cannot be a soliton in the sense
of Theorem \ref{APS:existence}.
\end{theorem}
\begin{proof} We still drop the subscript $c$. Now let $\phi$ be a smooth, radial function satisfying $0\leq \phi \leq 1$, $\phi(x)=1$ for $|x|\leq 1$, and $\phi(x)=0$ for $|x|\geq 2$.
For some $R$, we define
\begin{align*}
V_R(t):=\int_{\R^3} \phi_R(x) |u(t,x)|^2\; dx, \quad
\phi_R(x)=R^2\phi\left(\frac{|x|^2}{R^2}\right).
\end{align*}

On one hand, we have
\begin{align*}
\partial_t V_R(t) = 4 \Im \int_{\R^3}
\phi'\left(\frac{|x|^2}{R^2}\right) x \cdot \nabla u (t,x)\;
\overline{u(t,x)} \; dx.
\end{align*}
Therefore, we have
\begin{align}\label{virial:derivative:upper bound}
\big|\partial_t V_R(t)\big| \lesssim R
\end{align}
for all $t\geq 0$ and $R>0$.

On the other hand, by Lemma \ref{L:virial} and  H\"{o}lder's inequality, we have
\begin{align*}
 & \partial^2_t V_R(t)
=4 \int_{\R^3} \phi_R''(r) \big|\nabla u
(t,x)\big|^2\;dx -\int_{\R^3} (\Delta^2 \phi_R)(x) |u(t,x)|^2 \; dx\\
& \qquad \quad - \frac43 \int_{\R^3} (\Delta \phi_R)(x) |u(t,x)|^6\; dx +
\int_{\R^3} (\Delta \phi_R)(x) |u(t,x)|^4\; dx \\
= & 4 \int_{\R^3} \left( 2 |\nabla u(t,x)|^2 -2|u(t,x)|^6 +\frac32
|u(t,x)|^4 \right)\; dx \\
+ & O\left(\int_{|x|\geq R} \left(|\nabla u(t,x)|^2  +|u(t,x)|^6 +
|u(t,x)|^4 \right) \;dx + \left(\int_{R\leq |x|\leq 2R}
|u(t,x)|^6 \;dx \right)^{1/3}\right)\\
= & 4 K\left(u(t)\right)
+  O\left(\int_{|x|\geq R} \left(|\nabla u(t,x)|^2  +
|u(t,x)|^4 \right) \;dx + \left(\int_{R\leq |x|\leq 2R}
|u(t,x)|^6 \;dx \right)^{1/3}\right).
\end{align*}

By Lemma \ref{uniform bound}, we
have
\begin{align*}
 4 K\left(u(t)\right)= &\; 4 \int_{\R^3} \left( 2 |\nabla u(t,x)|^2 -2|u(t,x)|^6 +\frac32
|u(t,x)|^4 \right)\; dx \\
\gtrsim & \min\left(6(m-E(u(t))), \frac{2}{3} \big\|\nabla u(t)
\big\|^2_{L^2} +  \frac12 \big\|u(t)\big\|^4_{L^4} \right) \\
\gtrsim &   \big\|\nabla u(t) \big\|^2_{L^2} +
\big\|u(t)\big\|^4_{L^4} \\
 \gtrsim &
E(u(t)),
\end{align*}
Thus, choosing $\eta>0$ sufficiently small and
$\displaystyle R:=C(\eta)$ and by Corollary
\ref{compact:almost periodicity}, we obtain
\begin{align*}
 \partial^2_t V_R(t)  \gtrsim E(u(t)) = E(u_0),
\end{align*}
which implies that for all $T_1>T_0$
\begin{align*} (T_1-T_0)E(u_0) \lesssim
R=C(\eta).
\end{align*}
Taking $T_1$
sufficiently large, we obtain a contradiction unless $u\equiv0$. But
$u\equiv 0$ is not consistent with the fact that
$\big\|u\big\|_{ST(\R)}=\infty$.
\end{proof}

%%%%%%%%%%%%%%%%%%%%%%%%%%%%%%%%%%%%%%%%%%%%%%%%%%%%%%%%%%%%%%%%%%%%%%%%%%%%%%%%%%%%%%%%%%%
%
%
%                                   Section references
%
%
%%%%%%%%%%%%%%%%%%%%%%%%%%%%%%%%%%%%%%%%%%%%%%%%%%%%%%%%%%%%%%%%%%%%%%%%%%%%%%%%%%%%%%%%%%%


\begin{thebibliography}{10}
\newcommand{\msn}[1]{\href{http://www.ams.org/mathscinet-getitem?mr=#1}{\sc MR#1}}


\bibitem{Aubin:Sharp contant:Sobolev}T.~Aubin, \emph{Probl\'emes isop\'erim\'etriques et espaces de Sobolev},
J. Diff. Geom., \textbf{11} (1976), 573--598.

\bibitem{BahG:NLW:proffile decomp}H.~Bahouri and P.~G\'{e}rard,
\emph{High frequency approximation of solutions to critical
nonlinear wave equations}, Amer. J. Math., \textbf{121} (1999),
131--175.

\bibitem{Bou:NLS:99}J.~Bourgain, \emph{Global well-posedness of defocusing 3D
critical NLS in the radial case}, J. Amer. Math. Soc.,
\textbf{12}(1999), 145--171.

\bibitem{Bou:NLS:book}J.~Bourgain, \emph{Global solutions of nonlinear Schr\"{o}dinger
equations}, Amer. Math. Soc. Colloq. Publ. \textbf{46}, Amer. Math.
Soc., Providence, 1999.

\bibitem{Caz:NLS:book}T.~Cazenave, \emph{Semilinear Schr\"{o}dinger equations}. Courant
Lecture Notes in Mathematics, Vol. \textbf{10}. New York: New York
University Courant Institute of Mathematical Sciences, 2003.

\bibitem{CKSTT04}J.~Colliander, M.~Keel, G.~Staffilani, H.~Takaoka and T.~Tao, \emph{Global existence and scattering for rough
solution of a  nonlinear Schr\"{o}dinger equation on
$\mathbb{R}^3$}, Comm. Pure Appl. Math., \textbf{57}(2004),
987--1014.

\bibitem{CKSTT07}J.~Colliander, M.~Keel, G.~Staffilani, H.~Takaoka,
and T.~Tao, \emph{Global well-posedness and scattering for the
energy-cirtical nonlinear Schr\"{o}dinger equation in
$\mathbb{R}^3$},  Ann. of Math., \textbf{166}(2007), 1--100.

\bibitem{Dod:NLS:higher dim}B.~Dodson, \emph{Global well-posedness and scattering for the defocusing, $L^2$-critical, nonlienar
Schr\"{o}dinger equation when $d\geq 3$}, arXiv:0912.2467v3.

\bibitem{Dod:NLS:two dim}B.~Dodson,  \emph{Global well-posedness and scattering for the defocusing, $L^2$-critical, nonlienar
Schr\"{o}dinger equation when $d= 2$}, arXiv:1006.1375.

\bibitem{Dod:NLS:one dim}B.~Dodson,  \emph{Global well-posedness and scattering for the defocusing, $L^2$-critical, nonlienar
Schr\"{o}dinger equation when $d=1$}, arXiv:1010.0040.

\bibitem{Dod:NLS:foc}B.~Dodson, \emph{Global well-posedness and scattering for the mass critical nonlienar
Schr\"{o}dinger equation with mass below the mass of the ground
state}, arXiv:1104.1114.

\bibitem{DuyM:NLS:ThresholdSolution} T.~Duyckaerts and F.~Merle,
\emph{Dynamic of threshold solutions for energy-critical NLS.}
GAFA., \textbf{18:6}(2009), 1787--1840.
%
%\bibitem{DuyMerle:NLW:ThresholdSolution}T.~Duyckaerts and F.~Merle,
%\emph{Dynamic of threshold solutions for energy-critical wave
%equation.} Int. Math. Res. Papers (2010), Vol. \textbf{2008},
%article ID rpn002, 67 pages, doi:10.1093/imrp/rpn002.
%
\bibitem{DuyR:NLS:ThresholdSolution} T.~Duyckaerts and S.~Roudenko, \emph{Threshold solutions for the focusing 3D cubic
Schr\"{o}dinger equation.} Revista. Math. Iber.,
\textbf{26:1}(2010), 1--56.

\bibitem{Fos:NLS:exotic}D.~Foschi, \emph{Inhomogeneous Strichartz estimates}, J. Hyper. Diff. Equat.,
\textbf{2}(2005), 1--24.

\bibitem{GinV:85:NLS}J.~Ginibre and G.~Velo, \emph{Scattering theory in the energy space for a class of nonlinear Schr\"{o}dinger equation},
J. Math. Pures Appl., 64(1985), 363--401.

\bibitem{GinV:85:NLKG subcritical}J.~Ginibre and G.~Velo, \emph{Time decay of finite energy solutions of the nonlinear Klein-Gordon
and Schr\"{o}dinger equation}, Ann. Inst. Henri Poincar\'{e},
Physique th\'{e}orique, \textbf{43:4}(1985), 399--442.

\bibitem{IbrMN:f:NLKG}S.~Ibrahim, N.~Masmoudi and K.~Nakanishi,
\emph{Scattering threshold for the focusing nonlinear Klein-Gordon
equation}, to appear in Analysis $\&$ PDE.

\bibitem{KeT98}M. Keel and T. Tao, \emph{Endpoint Strichartz estimates}, Amer. J.
Math. \textbf{120:5}(1998), 955--980.

\bibitem{KenM:NLS:GWP}
C.~E.~Kenig and F.~Merle, \emph{Global well-posedness, scattering
and blow up for the energy-critical, focusing, nonlinear
Schr\"odinger equation in the radial case}, Invent. Math.,
\textbf{166}(2006), 645--675.

\bibitem{KenM:NLW:GWP}
C.~E.~Kenig and F.~Merle, \emph{Global well-posedness, scattering
and blow-up for the energy critical focusing non-linear wave
equation}, Acta Math., \textbf{201:2}(2008), 147--212.

\bibitem{Ker:NLS:profile decomp}S.~Keraani, \emph{On the defect of compactness for the
Strichartz estimates of the Schr\"{o}dinger equations.} J. Diff.
Equat., \textbf{175}(2001) 353--392.

\bibitem{Ker:NLS:compactness}S.~Keraani, \emph{On the blow up phenomenon of the critical
Schr\"{o}dinger equation}, J. Funct. Anal., \textbf{265}(2006),
171--192.

\bibitem{KiTV:NLS:2d}R.~Killip, T.~Tao and M.~Visan, \emph{The cubic nonlinear
Schr\"{o}dinger equation in two dimensions with radial data}, J. Eur. Math. Soc., \textbf{11:6}(2009), 1203--1258.

\bibitem{KiV:en-NLS:high dim}R.~Killip and M.~Visan, \emph{The focusing energy-critical
nonlinear Schr\"{o}dinger equation in dimensions five and higher},
Amer. J. Math., \textbf{132:2}(2010), 361--424.


\bibitem{KiVZ:NLS:high dim}R.~Killip, M.~Visan and X.~Zhang, \emph{The mass-critical
nonlinear Schr\"{o}dinger equation with radial data in dimensions
three and higher}, Analysis $\&$ PDE. \textbf{1:2}(2008), 229--266.

%
%\bibitem{KriNS:e-critical NLW}J.~Krieger, K.~Nakanishi and
%W.~Schlag, \emph{Global dynamics away from the ground state for the
%energy-critical nonlinear wave equation}, to appear in Amer. J.
%Math..
%
%\bibitem{KriNakSch:1D KG}J.~Krieger, K.~Nakanishi and W.~Schlag,
%\emph{Global dynamics above the ground state energy for the
%one-dimensional NLKG equation}, to appear in Math. Z..

\bibitem{LiZh:NLS}D.~Li and X.~Zhang, \emph{Dynamics for the energy critical nonlinear Schr\"{o}dinger equation in high
dimensions}, J. Funct. Anal., \textbf{256:6}(2009), 1928--1961.

\bibitem{MiaoWX:Har:dynamic}C.~Miao, Y.~Wu and G.~Xu, \emph{Dynamics for the focusing, energy-critical nonlinear Hartree
equation}, preprint.

\bibitem{MiXZ07b}C.~Miao , G.~Xu and L.~Zhao, \emph{Global well-posedness, scattering and blow-up
for the energy-critical, focusing Hartree equation in the radial
case}, Colloquium Mathematicum, \textbf{114}(2009), 213--236.

\bibitem{MiXZ:NLS:radial NLS:mass critical sca}C.~Miao , G.~Xu and L.~Zhao, \emph{Global well-posedness and scattering for the
mass-critical Hartree equation with radial data}, J. Math. Pures
Appl., \textbf{91}(2009), 49--79.

\bibitem{Nak:NLKG low dim:subcrit}K.~Nakanishi, \emph{Energy scattering for nonlinear Klein-Gordon and Schr\"{o}dinger equations in spatial
dimensions $1$ and $2$}, J. Funct. Anal., \textbf{169}(1999),
201--225.

\bibitem{Nak:01:NLKG subcritical}K.~Nakanishi, \emph{Remarks on the energy scattering for nonlinear Klein-Gordon and Schr\"{o}dinger
equations}, Tohoku Math. J., \textbf{53}(2001), 285--303.



\bibitem{NakSch:cubic NLS:Rigidity}K.~Nakanishi and W.~Schlag,
\emph{Global dynamics above the ground state energy for the cubic
NLS equation in $3D$}, to appear in Calc. of Variations and PDE.


\bibitem{OgaTsu:Blowup:NLS:91}T.~Ogawa and Y.~Tsutsumi, \emph{Blow-up of $H^1$ solution for the nonlinear Schr\"{o}dinger equation}, J. Diff. Equat.,
\textbf{92}(1991), 317--330.

\bibitem{RyV05}E.~Ryckman and M.~Visan, \emph{Global well-posedness and
scattering for the defocusing energy-critical nonlinear
Schr\"{o}dinger equation in $\mathbb{R}^{1+4}$}, Amer. J. Math.,
\textbf{129}(2007), 1--60.

\bibitem{Talenti:best constant}G.~Talenti, \emph{Best constant in Sobolev inequality}, Ann. Mat. Pura. Appl. \textbf{110}(1976), 353--372.

\bibitem{tao:book}T.~Tao, \emph{Nonlinear dispersive equations, local and global
analysis}, CBMS. Regional conference Series in Mathematics,
\textbf{106}. Published for the Conference Board of the Mathematical
Science, Washington, DC; by the American Mathematical Society
Providence, RI, 2006.

\bibitem{TaoVZ:NLS:combined}T.~Tao, M.~Visan and X.~Zhang, \emph{The nonlinear Schr\"{o}dinger equation with combined power-type
nonlinearities}, Comm. PDEs, \textbf{32}(2007), 1281--1343.


\bibitem{Vi05}M.~Visan, \emph{The defocusing energy-critical nonlinear
Schr\"{o}dinger equation in higher dimensions}, Duke Math. J.,
\textbf{138:2}(2007), 281--374.

\bibitem{Zha:global:NLKG}J.~Zhang, \emph{Sharp conditions of global existence for nonlinear Schr\"{o}dinger and
Klein-Gordon equations}, Nonlinear Anal. T. M. A.,
\textbf{48:2}(2002), 191--207.

\bibitem{Zhang:NLS:06}X.~Zhang,
\emph{On Cauchy problem of 3D energy critical Schr\"{o}dinger
equation with subcritical perturbations}, J. Diff. Equat.,
\textbf{230:2}(2006), 422--445.
\end{thebibliography}
\end{document}